\documentclass[12pt,a4paper]{article}
\usepackage{a4,amsmath,amssymb,hyperref,amsthm,cite}

\allowdisplaybreaks[4]
\renewcommand{\title}[1]{\vspace{10mm}\noindent{\Large{\bf #1}}\vspace{8mm}}
\newcommand{\authors}[1]{\noindent{\large #1}\vspace{3mm}}
\newcommand{\address}[1]{{\itshape #1\vspace{2mm}}}
\newtheorem{Definition}{Definition}
\newtheorem{Proposition}[Definition]{Proposition}
\newtheorem{Lemma}[Definition]{Lemma}
\newtheorem{Corollary}[Definition]{Corollary}
\newtheorem{Remark}[Definition]{Remark}
\newtheorem{Theorem}[Definition]{Theorem}
\newcommand\D{\mathcal D}
\newcommand\A{\mathcal A}
\newcommand\B{\mathcal B}

\makeatletter
\def\section{\@startsection{section}{1}{\z@}{-3.25ex plus -1ex minus
  -.2ex}{1.5ex plus .2ex}{\normalfont\large\bfseries}}
\def\subsection{\@startsection{subsection}{1}{\z@}{-3.25ex plus -1ex
  minus -.2ex}{1.5ex plus .2ex}{\normalfont\itshape}}

\renewenvironment{thebibliography}[1]
       {\section*{References}\frenchspacing\small
        \begin{list}{[\arabic{enumi}]}
       {\usecounter{enumi}\parsep=2pt\topsep 0pt
       \settowidth{\labelwidth}{[#1]}
       \leftmargin=\labelwidth\advance\leftmargin\labelsep
       \rightmargin=0pt\itemsep=0pt\sloppy}}{\end{list}}

\makeatother

\begin{document}

\hspace*{\fill}

\begin{center}

\title{Spectral geometry of the  Moyal plane with \\[1mm] harmonic propagation}

\authors{Victor Gayral\footnote{victor.gayral@univ-reims.fr}  and 
Raimar Wulkenhaar\footnote{raimar@math.uni-muenster.de}}

\address{$^1$Laboratoire de Math\'ematiques
de l'Universit\'e de Reims and   
Laboratoire de Math\'ematiques et Applications de 
l'Universit\'e de Metz, France}

\address{$^2$Mathematisches Institut der Westf\"alischen 
Wilhelms-Universit\"at, M\"unster, Germany}

\vskip 2cm

\textbf{Abstract} \vskip 3mm 

\begin{minipage}{14cm}%
 We construct a `non-unital spectral triple of finite volume' out of
 the Moyal product and a differential square root of the harmonic
 oscillator Hamiltonian.  We find that the spectral dimension of this
 triple is $d$ but the KO-dimension is $2d$.  We add another
 Connes-Lott copy and compute the spectral action of the
 corresponding $U(1)$-Yang-Mills-Higgs model. We find that the `covariant
 coordinate' involving the gauge field combines with the Higgs field
 to a unified potential, yielding a deep unification of discrete and
 continuous parts of the geometry.
\end{minipage}

\end{center}

{\vskip 3cm}

\section{Introduction}

\indent

Unlike the compact (unital) case \cite{Connes:2008??} and until now,
there is for complete non-compact Riemannian spin manifolds no proper
{\it reconstruction theorem from a spectral point of view}. Thus the
question of the defining `axioms' for non-unital spectral triples is
not yet fully answered.  However, the basic and most important ideas
of modifications for the locally compact case are clear and appeared
already in Connes' founding paper \cite{Connes:1996gi}.  The case of
the ordinary Dirac operator of a locally compact complete Riemannian
spin manifold manifests that one cannot assume the resolvent of the
Dirac type operator underlying a locally compact (non-unital) spectral
triple to be a compact operator. The natural replacement is to ask
that the `localized resolvent', i.e.  the resolvent multiplied with an
element of the algebra, is a compact operator.  Another issue,
explained in depth in \cite{Gayral:2003dm}, is the choice of {\it a
 unitization} of the algebra.  This choice is constrained by the
orientability condition, which in the unital case (and with an
integral metric dimension) is the question of the existence of an
Hochschild cycle defining a volume form through the noncommutative
integral given by the Dixmier trace. Again, commutative but locally
compact examples show that this has to be a Hochschild cycle on a
specific unitization of the algebra we start with, but not on the
algebra itself. With these two main modifications (compactness of the
localized resolvent and existence of a preferred unitization), most of
the conditions for a non-unital spectral triple are easy to spell out.
Only the Poincar\'e duality remains unclear to formulate.  To help the
reader with this discussion, in the appendix we have reproduced the
modified conditions for non-unital spectral triples, as in given in
\cite{Gayral:2003dm} with the only modification that the metric and
KO-dimensions do not have to coincide, according to the recent
formulation of the standard model \cite{Connes:2006qv} and the
Podle\'s quantum sphere \cite{Dabrowski:2003??}.  Note that these
conditions are not far away from those given in \cite{Rennie:2003??}.
However, in \cite{Gayral:2003dm, Rennie:2003??} there is an extra
assumption of existence of a system of local (or quasi-local) units,
akin to the local structure of a non-compact manifold. These locality
assumptions have been fully removed in a more recent joint work of one
of us \cite{CGRS2}. However, in that work, the focus is on the index
theoretical side of the notion of a spectral triple, not on the
noncommutative generalization of a spin manifold. The definition
for a non-unital spectral triple given in \cite{CGRS2} is the minimal
one, ensuring a well-posed Fredholm index problem with a numerical
index computable by means of a local representative of the Chern
character in cyclic cohomology.

\medskip

The present article  is devoted to the study of a situation somehow in between
the compact (unital) and non-compact (non-unital) setting. Indeed, our
Dirac-type operator has compact resolvent alone, but it does not
reflect the  metric dimension. It is only the localized
resolvent which exhibits the correct metric dimension. We term this weird
situation as `non-unital spectral triple of finite volume'.
 This has, at least, one very nice feature, namely that the
spectral action can be defined
and computed in the usual way. The main motivation for this example
comes from noncommutative quantum field theory.

Because of easy computability, quantum field theory on the Moyal plane
is the most-studied toy model for noncommutative quantum field
theories.  The ultra-violet/infra-red mixing problems arising in these
models have been solved by one of us in
\cite{Grosse:2003aj,Grosse:2004yu} by the introduction of a modified
propagator associated with the harmonic oscillator Hamiltonian.  See
also \cite{Rivasseau:2005bh,Gurau:2005gd,Gurau:2006yc} for different
renormalization proofs. From a physics point of view, the most
fascinating property of this model is the behavior of its
$\beta$-function
\cite{Grosse:2004by,Grosse:2005da,Disertori:2006uy,Disertori:2006nq},
which makes it a candidate for non-perturbatively renormalizable
quantum field theory in  dimension four
\cite{Magnen:2007uy,Rivasseau:2007fr,Grosse:2009pa}.  We recommend
\cite{Rivasseau:2007ab} for review and introduction to the literature.

In \cite{Grosse:2007jy}, one of us has sketched a possible spectral
triple for Moyal space with harmonic oscillator potential.  However,
it became clear very soon that working out the mathematical details is
a non-trivial issue so that the simpler commutative case was studied
first \cite{Wulkenhaar:2009pv}. In this paper, we achieve the
construction of a spectral triple for a suitable algebra of functions
on $\mathbb R^d$ endowed with the Moyal product, together with a Dirac
operator which is a square root of the $d$-dimensional harmonic
oscillator Hamiltonian.  Few remarks are in order. Firstly, in the
same way as to find a differential (and not pseudo-differential)
square root of the ordinary Laplacian on $\mathbb R^d$ where one has
to go $2^{\lfloor d/2\rfloor}\times2^{\lfloor d/2\rfloor}$ matrices,
to find a differential square root of the $d$-dimensional harmonic
oscillator Hamiltonian one has to go $2^d\times2^d$ matrices---this is the main
observation in \cite{Grosse:2007jy}. The second important remark has
to do with the choice of the function algebra with Moyal
product. Indeed, there are many non-unital Fr\'echet algebras of
functions with Moyal product that one may use while respecting most of
the non-unital spectral triple conditions. But there is only one for
which the finiteness axiom is satisfied for a Dirac-type operator
given by a square root of the harmonic oscillator Hamiltonian, namely
the algebra of Schwartz functions $\mathcal S(\mathbb R^d)$. A similar
phenomenon appeared in \cite{Gayral:2003dm} where it has been shown
that with the ordinary Dirac operator of $\mathbb R^d$, there is only
one choice of algebra of functions with Moyal product for which the
finiteness axiom is satisfied, namely the $L^2$-Sobolev space
$W^{2,\infty}(\mathbb R^d)$. Lastly, the construction of a Hochschild
cycle satisfying the orientability axiom requires (see also
\cite{Wulkenhaar:2009pv}) \emph{two different differential square
roots} of the harmonic oscillator Hamiltonian, not only one.

The paper is organized as follow. In Section \ref{ST}, we introduce
two spectral triples with common algebra $\A_\star$ given by the set
of Schwartz functions $\mathcal S(\mathbb R^d)$ with Moyal product,
and two different differential square roots of the harmonic oscillator
Hamiltonian acting densely on $\mathcal H:=L^2(\mathbb R^d)\otimes
\mathbb C^{2^d}$. The rest of the section is then devoted to prove
that these spectral triples are regular, that the metric dimension is
$d$,  the KO-dimension is $2d$ and that the 
dimension spectrum ${\rm Sd}$ is $d-\mathbb N$.  In
section \ref{SA}, we specialize to the case $d=4$ and after having
proven a heat-kernel expansion result adapted to our particular
situation, we explicitly compute the spectral action for a
$U(1)$-Higgs model.

\section{The harmonic oscillator spectral triple 
for Moyal space}
\label{ST}

We consider two Moyal-type deformations
$(\mathcal{A}_\star,\mathcal{D}_\bullet,\mathcal{H})$, $\bullet=1,2$, of the
(commutative) $d$-dimensional harmonic oscillator spectral triple
introduced in \cite{Wulkenhaar:2009pv}. In order to implement the
Moyal product, the dimension $d$ must be even. 

\subsection{An isospectral deformation}

On $L^2(\mathbb{R}^d)$,
we introduce the (unbounded) bosonic creation and annihilation operators 
$$
a_\mu := \partial_\mu + \tilde{\Omega} \,x_\mu \;,\qquad   
a_\mu^*: =-\partial_\mu + \tilde{\Omega}\, x_\mu\;,\qquad \mu=1,\dots,d\;,
$$ 
satisfying the commutation relations $[a_\mu,a_\nu]=
[a^*_\mu,a^*_\nu]=0$ and $[a_\mu,a_\nu^*]=
2\,\tilde{\Omega}\,\delta_{\mu\nu}$. Here, $\tilde \Omega > 0$ is a frequency 
parameter. On the exterior algebra $\bigwedge(\mathbb{C}^d)$, we introduce fermionic
partners $b_\mu,{b_{\mu}}^*$ which fulfill the anticommutation
relations $\{ b_\mu,b_\nu\}=\{{b_{\mu}}^*,{b_{\nu}}^*\}=0$ and $\{
b_\mu,{b_{\nu}}^*\}=\delta_{\mu\nu}$.  Then, on the Hilbert space
$$
\mathcal{H}:=L^2(\mathbb{R}^d) \otimes \bigwedge(\mathbb{C}^d)\simeq 
L^2(\mathbb R^d)\otimes
\mathbb C^{2^d}\;,
$$ 
these operators give rise to two selfadjoint
operators
\begin{align}
\label{D12}
\mathcal{D}_1:=\mathfrak{Q}_1+\mathfrak{Q}_1^*\;,\quad
\mathcal{D}_2:=i\mathfrak{Q}_2-i\mathfrak{Q}_2^*\;,
\end{align}
constructed out of the supercharges
\begin{align*}
\mathfrak{Q}_1:= a_\mu \otimes {b^{\mu}}^*\;,\qquad 
\mathfrak{Q}_2:= a_\mu \otimes b^\mu\;,
\end{align*}
where Einstein's summation  convention is used. Indices are raised
or lowered by the Euclidean metric $\delta^{\mu\nu}$ or
$\delta_{\mu\nu}$, respectively. The (anti-)commutation 
relations imply for $\bullet\in\{1,2\}$
\begin{align}
\label{H-x}
\mathcal{D}_\bullet^2 =
H \otimes 1 
-(-1)^\bullet \,\tilde{\Omega} \otimes \Sigma\;, \quad
H:= \tfrac{1}{2} \{a^\mu,a_\mu^*\}= 
-\partial_\mu\partial^\mu + \tilde{\Omega}^2 \,x_\mu x^\mu\;, \quad
\Sigma: = [b_\mu^*,b^\mu]\;.
\end{align}
We identify $H$ as the Hamiltonian of the $d$-dimensional harmonic
oscillator with frequency $\tilde\Omega$. Its spectrum is $\{\lambda_n =\tilde{\Omega}(\frac{d}{2}+n)\;,~
n\in \mathbb{N}\}$, where the eigenvalue $\lambda_n$ appears with 
multiplicity $\binom{n+d-1}{d-1}$. It then follows that
$(|\mathcal{D}_\bullet|+1)^{-z}$, $\bullet=1,2$, is trace-class for $\Re(z)>2d$. 
\begin{Remark}{\rm 
Our choice of $\mathcal{D}_2$ differs from \cite{Wulkenhaar:2009pv}.
One should take $\mathcal{D}_2$ from (\ref{D12}) also for the
commutative case to view our spectral triple as isospectral
deformation. As seen in the next, the choice (\ref{D12}) is required by the orientability
axiom. The commutative version is somehow degenerate and does not
detect the sign of the frequency in (\ref{H-x}).}
\end{Remark}

We now wish to implement the Moyal product $\star$ in this picture:
\begin{align}
\label{Moyal}
f\star g(x) = \int_{\mathbb{R}^d\times \mathbb{R}^d} \frac{dy
\,dk}{(2\pi)^d} \; 
f(x{+}\tfrac{1}{2}
\Theta \cdot k) \,g(x{+}y)\, \mathrm{e}^{i\langle k,y\rangle} \;,
\end{align}
parametrized by an invertible skew-symmetric matrix
$\Theta^t=-\Theta\in M_d(\mathbb{R})$.  We first need to find out
which algebra $\A_\star$ of functions (or distribution) with Moyal
product to use.  For that aim, observe that the finiteness condition
alone dictates the choice of the topological vector space underlying
the algebra $\A_\star$.  Indeed, from (\ref{H-x}) we conclude for $\bullet=1,2$
\begin{align*}
\mathcal{H}_\infty :=\bigcap_{m \geq 0} \mathrm{dom}(\mathcal{D}_\bullet^n) 
=\mathcal{S}(\mathbb{R}^d) \otimes \bigwedge (\mathbb{C}^d)
\simeq 
\mathcal{S}(\mathbb{R}^d,
\mathbb C^{2^d})\;,
\end{align*}
which is required to be a finitely generated projective module over
the algebra of the spectral triple. Thus this naturally leads us to
the choice $\mathcal{A}_\star
:=\big(\mathcal{S}(\mathbb{R}^d),\star\big)$, with even $d$.
\begin{Remark}{\rm
For the ordinary Dirac operator on the trivial spin bundle of $\mathbb
R^d$, the set of smooth spinors is isomorphic to $W^{2,\infty}(\mathbb
R^d)\otimes\mathbb C^{2^{\lfloor d/2\rfloor}}$. In this case, the topological
vector space underlying the choice of the algebra is the $L^2$-Sobolev
space $W^{2,\infty}(\mathbb R^d)$. Note that the latter is stable
under the Moyal product too, and that this is the choice made in
\cite{Gayral:2003dm}.}
\end{Remark}
Here are the main properties of the Moyal product we will use latter
on (for more information see~\cite{Gayral:2003dm}). First is strong
closedness
\begin{equation}
\label{tracial}
\int f\star g(x) \,dx=\int f(x)\,g(x)\,dx=\int g\star f(x)\,dx\,,
\quad\forall f,g\in L^2(\mathbb R^4) \;,
\end{equation}
then, we have the Leibniz rule 
\begin{equation}
\label{leibniz}
\partial_\mu(f\star g)=\partial_\mu f\star g+f\star\partial_\mu g\;,
\end{equation}
and the following identities
\begin{equation}
\label{identities}
\{f,x^\mu\}_\star:=x^\mu\star f + f\star x^\mu
=2x^\mu f,\quad[x^\mu ,f]_\star:=x^\mu\star f-f\star x^\mu
=i\Theta^{\mu\nu} \partial_\nu f\;,
\end{equation}
both holding for $f,g\in\mathcal A_\star$. Last is the (non-unique)
factorization property \cite[p.~877]{Phobos}
\begin{equation}
\label{factor}
\forall f\in\A_\star\,,\quad\exists\, g,h\in\A_\star\,:\, f=g\star h\;.
\end{equation}

Following~\cite{Gayral:2003dm}, we then specify the preferred
unitization $\mathcal{B}_\star$ of $\mathcal{A}_\star$, as the
\emph{space of smooth bounded functions on $\mathbb{R}^d$ with all
 partial derivatives bounded}. The Moyal product (\ref{Moyal})
extends to $\mathcal{B}_\star$, and $\mathcal{A}_\star \subset
\mathcal{B}_\star$ is an essential two-sided ideal, \cite[Theorem
 2.21]{Gayral:2003dm}, but is not dense. The reason why we chose
this particular unitization is that $\mathcal{B}_\star$ contains the
plane waves  and constant functions (but no other
non-constant polynomials) and this is crucial for the orientability
condition (see subsection \ref{orientability}).
According to \cite[Theorem 2.21]{Gayral:2003dm},
the $C^*$-completion of $\mathcal B_\star$ is
\begin{align*}
A_\star:= \big\{ T\in \mathcal{S}'(\mathbb{R}^d)\;:~ 
T\star f \in L^2(\mathbb{R}^d) \text{ for all } f \in 
L^2(\mathbb{R}^d)\big\}\;.
\end{align*}
Therefore, $A_\star$ acts on $\mathcal{H}$ by
componentwise left Moyal multiplication, that we denote by  $L_\star$:
\begin{align*}
L_\star: {A}_\star \times \mathcal{H} \to \mathcal{H}\;,\qquad 
(f,\psi \otimes m) \mapsto   (f\star
\psi)\otimes m\;,
\end{align*}
for $\psi \in L^2(\mathbb{R}^d)$ and $m\in \bigwedge(\mathbb{C}^d)$.
In particular, we have the bounds
\cite{Gayral:2003dm}:
$$
\|L_\star(f)\|\leq C_1(\Theta)\,\|f\|_2,\quad f\in\mathcal A_\star\,,\qquad 
\|L_\star(f)\|\leq C_2(\Theta)\,\sup_{|\alpha|\leq d+1}
\|\partial^\alpha f\|_\infty,\quad f\in\mathcal B_\star\;.
$$
We also define the (anti-)action
$R_\star$ of $A_\star$ on $\mathcal{H}$ by
componentwise right Moyal multiplication:
\begin{align*}
R_\star: \mathcal{A}_\star \times \mathcal{H} \to \mathcal{H}\;,\qquad 
(f,\psi\otimes m) \mapsto   (\psi\star
f)\otimes m\;.
\end{align*}
Since the complex conjugation is an involution of the algebra
$A_\star$, and from the traciality of the Moyal product
\eqref{tracial}, we get $L_\star(f)^*=L_\star(\bar f)$,
$R_\star(f)^*=R_\star(\bar f)$. Moreover, this also shows that the two
representations $L_\star$ and $R_\star$ are isometric:
$$
\|L_\star(f)\|=\|R_\star(f)\|\,,\quad \forall f\in A_\star\;.
$$ To avoid too many notations, $L_\star$, $R_\star$ will also denote
the left and right actions of $\A_\star$ and $\mathcal B_\star$ on
$L^2(\mathbb R^d)$.

We now check that our spectral triple $(\A_\star,\mathcal H,\mathcal
D_\bullet)$, $\bullet=1,2$, defines a non-unital spectral triple with spectral
dimension $d$ and KO-dimension  $2d$, in the sense of Definition
\ref{Def:ST} in the Appendix.

\subsection{Boundedness and compactness }

From \eqref{identities}, we obtain for $f\in \mathcal{B}_\star$
on $\mathrm{dom}(\mathcal{D}_\bullet) $:
\begin{align}
\label{Gamma}
[\mathcal{D}_1, L_\star (f)]&= 
L_\star(i \partial_\mu f)\otimes \Gamma^\mu\;,\qquad
&
\Gamma^\mu &:= 
(i b^\mu-ib^{*\mu})- \tfrac{1}{2}\tilde{\Omega}\,\Theta^{\mu\nu}\,
(b_\nu+b_\nu^*)\;,
\nonumber
\\
[\mathcal{D}_2, L_\star (f) ] &= 
L_\star(i \partial_\mu f)\otimes \Gamma^{\mu+d}\;, &
\qquad
\Gamma^{\mu+d} &:= 
(b^\mu+b^{*\mu})- \tfrac{1}{2}\tilde{\Omega}\,\Theta^{\mu\nu}\,
(ib_\nu-ib^*_\nu)\;.
\end{align}
As $\partial_\mu f \in \mathcal{B}_\star$, the commutator
$[\mathcal{D}_\bullet,L_\star(f)]$ extends to a bounded operator.  It is
a remarkable property of the Moyal algebra that just the $d$-dimensional
differential of $f$ appears, no $x$-multiplication.

For  the compactness condition, there is not much to say as
$(\mathcal{D}_\bullet+\lambda)^{-1}$ is already a compact operator on
$\mathcal{H}$.  Then, $L_\star(f) (\mathcal{D}_\bullet+\lambda)^{-1}$ is
compact for any $f \in \mathcal{A}_\star$, even for $f \in
\mathcal{B}_\star$. 

\subsection{Orientability}
\label{orientability}

Note first that the operators $\Gamma^\mu,\Gamma^{\mu+d}$ defined by
(\ref{Gamma}) satisfy the anticommutation relations
$$
\{ \Gamma^\mu,\Gamma^\nu\} =\{ \Gamma^{\mu+d},\Gamma^{\nu+d}\} = 2\,
(g^{-1})^{\mu\nu}\;,\quad \{ \Gamma^\mu,\Gamma^{\nu+d}\}=0\;,
$$
where the symmetric matrix $g\in{\rm GL}(d,\mathbb R)$ is defined by
\begin{align}
\label{metric}
g:= \big({\rm Id}_d-\tfrac{1}{4}\tilde{\Omega}^2\,
\, \Theta^2\big)^{-1}\;,
\end{align}
and plays the role of a effective metric.  Note that
$\Theta^2=-\Theta^t\Theta$ is negative definite so that 
$$
(1+\tfrac{1}{4}\tilde{\Omega}^2\|\Theta\|^2)^{-1}{\rm Id}_d\leq g \leq {\rm Id}_d\;.
$$
 We will frequently use that
$\Theta,g,g^{-1}$ commute with each other.  Raising and lowering of
summation indices will always be performed with the Euclidean metric
$\delta^{\mu\nu},\delta_{\mu\nu}$.

Thus the $\{\Gamma^1,\dots,\Gamma^{2d}\}$ generate a Clifford algebra
of double dimension $2d$. The inverse transformation of (\ref{Gamma})
reads
\[
ib_\nu-ib^*_\nu
= g_{\nu\mu} \big(\Gamma^\mu+\tfrac{1}{2}\tilde{\Omega}\,\Theta^{\mu\rho}\,
\Gamma_{\rho+d}\big)\;,\qquad
b_\nu+b^*_\nu 
=g_{\nu\mu} \big(\Gamma^{\mu+d}+\tfrac{1}{2}\tilde{\Omega}\,\Theta^{\mu\rho}\,
\Gamma_{\rho}\big)\;.
\]
Therefore, we can express $\mathcal{D}_1$, in terms
of $\{\Gamma^1,\dots,\Gamma^{2d}\}$, but not in terms of the half set
of operators $\{\Gamma^1,\dots,\Gamma^d\}$ produced by the commutator
of $\mathcal{D}_1$ with $\mathcal{B}_\star$. Similar comments apply for
$\mathcal{D}_2$. In conclusion, we get
\begin{align}
\label{DGamma}
\mathcal{D}_1 &= i\partial^\nu \otimes g_{\nu\mu}
\big(\Gamma^\mu+\tfrac{1}{2}\tilde{\Omega}\,\Theta^{\mu\rho}\,
\Gamma_{\rho+d}\big) 
+ \tilde{\Omega}\, x^\nu \otimes g_{\nu\mu}
\big(\Gamma^{\mu+d}+\tfrac{1}{2}\tilde{\Omega}\,\Theta^{\mu\rho}\,
\Gamma_{\rho}\big) \;,
\nonumber
\\*
\mathcal{D}_2 &= i\partial^\nu \otimes g_{\nu\mu}
\big(\Gamma^{\mu+d}+\tfrac{1}{2}\tilde{\Omega}\,\Theta^{\mu\rho}\,
\Gamma_{\rho}\big) 
+ \tilde{\Omega} \,x^\nu\otimes  g_{\nu\mu}
\big(\Gamma^{\mu}+\tfrac{1}{2}\tilde{\Omega}\,\Theta^{\mu\rho}\,
\Gamma_{\rho+d}\big) \;. 
\end{align}
General results for Clifford algebras then show that
any element of the Clifford algebra which anticommutes with \emph{every}
$\Gamma^\mu$ and $\Gamma^{\mu+d}$ is a multiple of the
anti-symmetrized product of the generators
$\{\Gamma^1,\dots,\Gamma^{2d}\}$. Therefore, a grading operator
commutating with $\mathcal{D}_1$ cannot be found in the algebra generated
by $L_\star(f),R_\star(f)$ and $[\mathcal{D}_1,L_\star(f)]$, so that
an implementation of the orientability axiom requires both Dirac-type
operators $\mathcal{D}_1,\mathcal{D}_2$. 

Let $u_\mu:=e^{-ix_\mu}\in \mathcal{B}_\star$.
We know from\cite{Gayral:2003dm} that the element
\begin{align*}
\boldsymbol{c} := \sum_{\sigma \in S_d} \epsilon(\sigma)
\frac{\mathrm{i}^{\frac{d(d-1)}{2}}\sqrt{\det g}}{d!}
\big(
(u_1 \star \cdots \star u_d)^{-1}\otimes 1\big) 
\otimes u_{\sigma(1)} \otimes \cdots 
u_{\sigma(d)} \;,
\end{align*}
is a Hochschild $d$-cycle for the algebra $\mathcal B_\star$, with
values in $\mathcal B_\star\otimes \mathcal B_\star^{o}$. (In the
expression of $\boldsymbol{c}$, the inverse is with respect to the
$\star$-product and the scaling by $\sqrt{\det g}$ is irrelevant for
cyclicity.)  For $\pi_{\mathcal{D}_\bullet}$ defined in the Appendix, 
we then obtain from (\ref{Gamma})
\begin{align}
\boldsymbol{\gamma}_1 
&:= \pi_{\mathcal{D}_1}(\boldsymbol{c}) = \frac{\sqrt{\det g}}{d!} \, \mathrm{i}^{\frac{d(d-1)}{2}} \sum_{\sigma \in S_d} 
\epsilon(\sigma)
\otimes 
\Gamma^{\sigma(1)}\cdots  \Gamma^{\sigma(d)}\;,
\nonumber
\\ 
\boldsymbol{\gamma}_2 &:= 
\pi_{\mathcal{D}_2}(\boldsymbol{c})  
= \frac{\sqrt{\det g}}{d!}\,
\mathrm{i}^{\frac{d(d-1)}{2}}\sum_{\sigma \in S_d} 
\epsilon(\sigma)  \otimes 
\Gamma^{\sigma(1)+d}\cdots  \Gamma^{\sigma(d)+d}\;,
\end{align}
and they satisfy the relations:
\begin{align*}
\boldsymbol{\gamma}_1^2=1=
\boldsymbol{\gamma}_2^2 \,,\qquad
\boldsymbol{\gamma}_1^*=\boldsymbol{\gamma}_1\,,\quad
\boldsymbol{\gamma}_2^*=\boldsymbol{\gamma}_2\,,\qquad 
\boldsymbol{\gamma}_1\boldsymbol{\gamma}_2=(-1)^d
\,\boldsymbol{\gamma}_2\boldsymbol{\gamma}_1\;.
\end{align*}
Thus we define
\begin{align}
\label{grading}
\Gamma :=(-i)^d \boldsymbol{\gamma}_1 \boldsymbol{\gamma}_2 \;.
\end{align}
Since $\boldsymbol{\gamma}_1,\boldsymbol{\gamma}_2$ commute with every
element of $\mathcal{A}_\star$ or $\mathcal{B}_\star$, $\Gamma$ does
too and the discussion above shows that
$$
\Gamma^2=1\,,\qquad\{\mathcal{D}_\bullet,\Gamma\}=0\,,\quad \bullet=1,2\;,
$$
so that $\Gamma$ defines the grading operator for the two spectral
triples $(\A_\star,\mathcal H,\mathcal D_\bullet)$, $\bullet=1,2$.  We
stress that the necessity of the two Dirac operators
$\mathcal{D}_1$,$\mathcal{D}_2$, is quite different from conventional
spectral triples \cite{Connes:2008??} where a single operator is
needed.

Note also that from the explicit formulae of $\mathcal{D}_\bullet$,
(up to a possible sign) one has the relation $\Gamma=1 \otimes
(-1)^{N_f}$ in terms of the fermionic number operator
$N_f=b_\mu^*b^\mu$.

\subsection{KO-dimension and other algebraic conditions}

The real structure is an anti-linear isometry $J$ on $\mathcal{H}$.
We assume that for $d$ even the KO-dimension $k$ is even, too. Then,
according to the sign table in the Appendix we have 
\[
J \mathcal{D}_\bullet = \mathcal{D}_ \bullet J\,,\quad\bullet=1,2\;.
\]
This is achieved by the following non-trivial action on the matrix part of
$\mathcal H$:
\begin{align}
\label{Jab}
J a_\mu J^{-1}= a_\mu\;,\qquad 
J a_\mu^* J^{-1}= a_\mu^*\;,\qquad 
J b_\mu J^{-1}=  b_\mu^*\;,\qquad 
J b_\mu^* J^{-1}= b_\mu\;.
\end{align}
In particular, conjugation by $J$ preserves the (anti-)commutation
relations. We can view $\bigwedge (\mathbb{C}^d)$ as generated by
repeated action of $\{b^\dag_\mu\}$ on the vacuum vector $|0\rangle$
defined by $b_\mu|0\rangle=0$.
It then follows that, up to a prefactor of modulus 1, which cancels in
every relation of the dimension table, $J$ is the
Hodge-$*$ operator on $\bigwedge (\mathbb{C}^d)$, i.e.\ is uniquely
defined by 
\[
J|0\rangle = b_1^*b_2^*\cdots b_d^* |0\rangle\;,
\]
together with (\ref{Jab}) and the anti-linearity $J (z \psi)=\bar{z}
J\psi$. In particular,
$J \circ L_\star(f) \circ J^{-1} =R_\star(\overline{f})$, which
implements the opposite algebra and 
achieves the order-one condition:
\begin{align}
[J L_\star(f_1)J^{-1}, L_\star(f_2)]&=0\;, &
[J L_\star(f_1)J^{-1}, [\mathcal{D}_\bullet,L_\star(f_2)]]&=0\;,\qquad
\text{for all $f_1,f_2 \in \mathcal{B}_\star$}\;.
\end{align}
To compute $J^2$ we consider, for $\mu_1<\mu_2<\dots<\mu_k$, 
\[
J (b_{\mu_1}^*\cdots b_{\mu_k}^*|0\rangle)
= b_{\mu_1}\cdots b_{\mu_k}
b_1^*b_2^*\cdots b_d^*
|0\rangle= (-1)^{\sum_{j=1}^k (\mu_j-1)}
b_1^*\stackrel{\mu_1\dots \mu_k}{\check{\dots}} b_d^* |0\rangle\;.
\]
The notation $\stackrel{\mu_1\dots \mu_k}{\check{\dots}}$ 
means that $b_{\mu_1}^*,\dots,b_{\mu_k}^*$ are missing. 
We apply $J$ again, to get:
\[
J^2 (b_{\mu_1}^*\cdots b_{\mu_k}^*|0\rangle)
=(-1)^{\sum_{j=1}^k (\mu_j-1)}
J(b_1^*\stackrel{\mu_1\dots \mu_k}{\check{\dots}} b_d^* |0\rangle)
=(-1)^{\sum_{j=1}^d (j-1)} b_{\mu_1}^*\cdots b_{\mu_k}^*|0\rangle\;,
\]
which means
\begin{align}
J^2 =(-1)^{\frac{d(d-1)}{2}} \;.
\end{align}
From (\ref{Gamma}) it follows that $J$ commutes with $\Gamma^\mu$ and
$\Gamma^{\mu+d}$. From (\ref{grading}) we then conclude 
$J\Gamma=(-1)^d \Gamma J$. Comparing these results with the dimension
table in the Appendix, we have proven:
\begin{Proposition}
The spectral geometries
$(\mathcal{A}_\star,\mathcal{H},\mathcal{D}_\bullet,\,\Gamma,J)$, $\bullet=1,2$,
for the $d$-dimensional Moyal algebra $\mathcal{A}_\star$ are of
$\rm KO$-dimension $2d \mod 8$.  
\end{Proposition}

\subsection{Metric dimension}

Since $\D_\bullet$, $\bullet=1,2$,
squares (up to matrices) to the $d$-dimensional harmonic oscillator
Hamiltonian, we already now that $(1+\D_\bullet^2)^{-d}$ belongs to the
Dixmier ideal $\mathcal L^{1,\infty}(\mathcal H)$. In this subsection
we are going to prove that for the localized operators, the critical
dimension is reduced by a factor of $2$, that is for all
$f\in\A_\star$, the operators $L_\star(f)(1+\D_\bullet^2)^{-d/2}$ belong
to $\mathcal L^{1,\infty}(\mathcal H)$ and that any of its Dixmier traces
is a constant multiple of the integral of $f$.  To obtain both Dixmier
traceability and the value of the Dixmier trace, we will use the results of
\cite{CGRS1}.  In order to do this, we need some
preliminary Lemmas (which will also be needed to check the regularity
condition, to obtain the dimension spectrum and to compute the
spectral action).

\begin{Lemma}
\label{nabla-com}
Introducing the operators on $L^2(\mathbb R^d)$:
\begin{align*}
\nabla_\mu := 
\partial_\mu +\tfrac{1}{2}i\,\tilde{\Omega}^2\, 
\Theta_{\mu\nu} \,x^\nu\;,
\qquad
\tilde{\nabla}_\mu := \tfrac{1}{2}
\big(\partial_\mu - 2i \,(\Theta^{-1})_{\mu\nu}\, x^\nu\big)\;,
\qquad \mu=1,\cdots,d\;,
\end{align*}
we have the following relations for $f\in \mathcal B_\star$:
$$
[H,L_\star (f)] = -
L_\star\big((g^{-1})^{\mu\nu} \partial_\mu \partial_\nu f\big) 
- 2  L_\star(\partial^\mu f) \nabla_\mu\;,
$$
$$
[\nabla^{\mu},L_\star(f)] = L_\star\big((g^{-1})^{\mu\nu}\partial_\nu
f\big)\;, 
\qquad [\tilde{\nabla}^\mu,L_\star(f)] = L_\star(\partial_\mu f)\;, 
$$
$$
[H,\nabla_\mu] = -2i \tilde{\Omega}^2 \Theta_{\mu\nu} 
\tilde{\nabla}^\nu \;, \qquad
[H,\tilde{\nabla}_\mu] = 2i (\Theta^{-1})_{\mu\nu} \nabla^\nu \;, 
$$
$$
[\tilde{\nabla}_\mu,\tilde{\nabla}_\nu]  = i(\Theta^{-1})_{\mu\nu} \;, 
\qquad [\nabla_\mu ,\nabla_\nu] =
- i\tilde{\Omega}^2 \Theta_{\mu\nu} \;,
\qquad
[\nabla^\mu  ,\tilde{\nabla}_\nu] = 
i (g^{-1})^{\mu\rho} (\Theta^{-1})_{\rho\nu} \;.
$$
\end{Lemma}
\begin{proof}
This  follows from the relations \eqref{H-x}, \eqref{leibniz} and 
\eqref{identities}.
\end{proof}

\begin{Corollary}
\label{nabla-com2}
Let $P_\alpha(\hat{\nabla})$ be an element of order $\alpha$ of the
polynomial algebra generated by $\nabla,\tilde{\nabla}$.
Then
$P_\alpha(\hat{\nabla})(1+H)^{-\alpha/2}$ extends to a bounded 
operator.
\end{Corollary}
\begin{proof}
From the operator inequalities (no summations on $\mu$ but summation on $\nu$)
$$
|(1+H)^{-1/2}\partial_\mu|^2=-\partial_\mu(1+H)^{-1}\partial_\mu
\leq -\partial_\mu(1-\partial_\nu\partial^\nu)^{-1}\partial_\mu\;,
$$
and 
$$
|(1+H)^{-1/2}x_\mu|^2=x_\mu(1+H)^{-1}x_\mu\leq
x_\mu(1+\tilde{\Omega}^2 x_\nu x^\nu)^{-1}x_\mu\;,
$$
we see that $\hat{\nabla}_\mu(1+H)^{-1/2}$ is bounded. Then, the general
case follows by induction using
$$
\hat{\nabla}_{\mu_1}\hat{\nabla}_{\mu_2}(1+H)^{-1}
=\hat{\nabla}_{\mu_1}(1+H)^{-1} \hat{\nabla}_{\mu_2}+
\hat{\nabla}_{\mu_1}[\hat{\nabla}_{\mu_2},(1+H)^{-1}]\;,
$$
and
$$
\hat{\nabla}_{\mu_1}[\hat{\nabla}_{\mu_2},(1+H)^{-1}]
=-\hat{\nabla}_{\mu_1}(1+H)^{-1}[\hat{\nabla}_{\mu_2},H](1+H)^{-1}\;,
$$
which is bounded, too, according to Lemma \ref{nabla-com}.
\end{proof}

The following Proposition will be crucial for the computation of the
spectral action, the estimate we need to evaluate the Dixmier trace
and to compute the dimension spectrum and the residues of the
associated zeta functions.
\begin{Proposition}
\label{Proposition-trace}
For $f\in \mathcal B_\star$, define $\mathcal{T}_{\mu_1\dots\mu_k} (f) 
:= \mathrm{Tr} \big( L_\star(f) \nabla_{\mu_1}\dots\nabla_{\mu_k}
e^{-tH}\big)$. Then  one has
\begin{align*}
\mathcal{T}_{\mu_1,\dots \mu_k}(f)
&=\sum_{1\leq j_1 <j_2 < \dots <j_{2a} \leq k} 
\Big(\frac{\tilde{\Omega}}{2\pi\sinh(2\tilde{\Omega}
t)}\Big)^{\frac{d}{2}}  
\nonumber
\\
&\quad\times \int_{\mathbb{R}^d} dz\,\sqrt{\det g}\; f(z)\,
e^{-\tilde{\Omega}\tanh(\tilde{\Omega}t) \langle z,gz\rangle}\,
(\mathcal{Z}_{\mu_1}
\stackrel{j_1\dots j_{2a}}{\check{\dots}} \mathcal{Z}_{\mu_k})
(\mathcal{N}_{\mu_{j_1}\mu_{j_2}} \cdots 
\mathcal{N}_{\mu_{j_{2a-1}}\mu_{j_{2a}}}) \;,
\end{align*}
where
\begin{align*}
\mathcal{Z}_\mu &:=  - \tilde{\Omega} \tanh(\tilde{\Omega}t) z_\mu 
+ i\tilde{\Omega}^2 (\Theta gz)_\mu\;,
\\
\mathcal{N}_{\mu\nu} &:=-
\tfrac12\tilde{\Omega}\,(\coth(\tilde{\Omega}t)+\tanh(\tilde{\Omega}t))
(g^{-1})_{\mu\nu} 
-\tfrac12\tilde{\Omega}^3\,\coth(\tilde{\Omega}t) 
(\Theta g\Theta)_{\mu\nu}
-\tfrac12i\,\tilde{\Omega}^2\,  \Theta_{\mu\nu}\;,
\end{align*}
and $
\stackrel{j_1\dots j_{2a}}{\check{\dots}}$
means that $\{\mathcal{Z}_{\mu_{j_1}},\dots
\mathcal{Z}_{\mu_{j_{2a}}}\}$ are missing in the product 
$\mathcal{Z}_{\mu_1} \cdots \mathcal{Z}_{\mu_k}$.
(Remember that $g$ is a constant metric so that $\sqrt{\det g}$ can
also be taken in front of the integral).

In particular,
\[
\mathrm{Tr} 
\big( e^{-tH}\big)
=\big(2 \sinh (\tilde{\Omega} t)\big)^{-d}\;.
\]
\end{Proposition}
\begin{proof}
Since  $L_\star(f) \nabla_{\mu_1}\dots
\nabla_{\mu_k}e^{-tH}$ is trace-class 
(because  $\nabla_{\mu_1}\dots\nabla_{\mu_k}
e^{-tH/2}$ is bounded by Corollary \ref{nabla-com2} and 
$e^{-tH/2}$ is trace-class),
the trace can be evaluated as the integral of the kernel on the diagonal.
Thus, in integral kernel representation, we have to compute

\begin{align*}
&\mathcal{T}_{\mu_1,\dots \mu_k}(f)=
\\
&
\int_{\mathbb{R}^d\times \mathbb{R}^d} dx \,dy\; 
(L_\star(f))(x,y) 
\Big(\frac{\partial}{\partial y^{\mu_1}}
+ \frac{i}{2}\tilde{\Omega}^2\Theta_{\mu_1\nu_1} y^{\nu_1}\Big)
\cdots 
\Big(\frac{\partial}{\partial y^{\mu_k}}
+ \frac{i}{2}\tilde{\Omega}^2\Theta_{\mu_k\nu_k} y^{\nu_k}\Big)
\big( e^{-tH}(y,x) \big)\;.
\end{align*}
The operator kernel of $e^{-t H}$ is the Mehler kernel
\begin{align}
\label{Mehler-n}
e^{-t H}(x,y) =  
\Big(\frac{\tilde\Omega}{2\pi\sinh(2\tilde\Omega t)}
\Big)^{d/2}
e^{-\frac{\tilde\Omega}{4} \coth(\tilde\Omega t) \|x-y\|^2 
-\frac{\tilde\Omega}{4}\tanh(\tilde\Omega t) \|x+y\|^2 }\;,
\end{align}
while  the operator kernel of $L_\star(f)$  is readily identified to be
\begin{align}
\label{Lstarf}
L_\star (f)(x,y)=\frac{1}{\pi^d \det \Theta} \int d z  \,f(z)
\,\mathrm{e}^{i\langle x-y,\Theta^{-1} (x+y)\rangle
+2i\langle z,\Theta^{-1} (x-y)\rangle}\;.
\end{align}
We introduce $u=x-y$ and $v=x+y$ and 
\begin{align*}
\mathfrak{D}_\mu(u,v) &:=  \frac{\tilde{\Omega}}{2}\coth(\tilde{\Omega}t)
u_\mu - \frac{\tilde{\Omega}}{2}\tanh(\tilde{\Omega}t)v_\mu 
+\frac{i}{4}\tilde{\Omega}^2 \Theta_{\mu\alpha} (v^\alpha-u^\alpha)\;,
\\
\mathcal{Y}_{\mu\nu} & := -
\frac{\tilde{\Omega}}{2}(\coth(\tilde{\Omega}t)+\tanh(\tilde{\Omega}t))
\delta_{\mu\nu} -\frac{i}{2}\tilde{\Omega}^2 \Theta_{\mu\nu}\;,
\end{align*}
to obtain
\begin{align*}
&\mathcal{T}_{\mu_1,\dots \mu_k}(f)
\nonumber
\\*
&=\sum_{1\leq j_1 <j_2 < \dots <j_{2a} \leq k} 
\Big(\frac{\tilde{\Omega}}{2\pi\sinh(2\tilde{\Omega}
t)}\Big)^{\frac{d}{2}} \frac{1}{(2\pi)^d \det \Theta}
\nonumber
\\
& \times \int \!\!\! du\,dv\,dz\; f(z)
\mathfrak{D}_{\mu_1}(u,v)\stackrel{j_1\dots j_{2a}}{\check{\dots}} 
\mathfrak{D}_{\mu_k}(u,v) 
\mathcal{Y}_{\mu_{j_1}\mu_{j_2}} \cdots 
\mathcal{Y}_{\mu_{j_{2a-1}}\mu_{j_{2a}}} 
e^{-\frac{1}{2}\langle
(u,v),Q(u,v)\rangle-\langle(u,v),(2i\Theta^{-1}z,0)\rangle}
\nonumber
\\*
&=\sum_{1\leq j_1 <j_2 < \dots <j_{2a} \leq k} 
\Big(\frac{\tilde{\Omega}}{2\pi\sinh(2\tilde{\Omega}
t)}\Big)^{\frac{d}{2}}  \frac{1}{\det \Theta \sqrt{\det Q}}
\nonumber
\\
&\times \int_{\mathbb{R}^d} dz\; f(z)
\mathfrak{D}_{\mu_1}(\tfrac{i\partial}{\partial \xi},
\tfrac{i\partial}{\partial \eta})
\stackrel{j_1\dots j_{2a}}{\check{\dots}} 
\mathfrak{D}_{\mu_k}(\tfrac{i\partial}{\partial \xi},
\tfrac{i\partial}{\partial \eta}) 
\mathcal{Y}_{\mu_{j_1}\mu_{j_2}} \cdots 
\mathcal{Y}_{\mu_{j_{2a-1}}\mu_{j_{2a}}} \mathcal{E}\Big|_{\xi=\eta=0}\;,
\end{align*}
where $\mathcal{E} := 
e^{-\frac{1}{2}\langle(-2z\Theta^{-1}+\xi,\eta),Q^{-1} 
(2\Theta^{-1}z+\xi,\eta)\rangle}$ and 
$Q \in M_{2d}(\mathbb{C})$ is given by
$$
Q=
\begin{pmatrix}
\frac{\tilde\Omega}{2} \coth(\tilde\Omega t)\,{\rm Id}_d&-i\Theta^{-1}\\
i\Theta^{-1}&\frac{\tilde\Omega}{2} \tanh(\tilde\Omega t)\,{\rm Id}_d
\end{pmatrix}\;.
$$
Recalling that $g^{-1}=1-\frac{\tilde{\Omega}^2}{4}\Theta^2$
and $g\Theta=\Theta g$, we find by writing $Q$ as product of triangle
matrices 
$$
\det Q = \frac{1}{\det g\, (\det \Theta)^2}\,,\qquad
Q^{-1}=\begin{pmatrix}
-\frac{\tilde\Omega}{2} \tanh(\tilde\Omega t)\,g\,\Theta^2&-ig\,\Theta\\
ig\,\Theta&- \frac{\tilde\Omega}{2} \coth(\tilde\Omega t)\,g\,\Theta^2
\end{pmatrix}\;,
$$
so that
\begin{align*}
\mathcal{E} =
\exp\Big\{& -\tilde{\Omega} \tanh(\tilde\Omega t) \langle z,gz\rangle
- 
\tilde{\Omega} \tanh(\tilde\Omega t) \langle z , g \Theta \xi \rangle  
- 2 i \langle z , g \eta\rangle 
\\*
&
 + \frac{\tilde{\Omega}}{4} \tanh(\tilde\Omega t) 
\langle \xi , \Theta g \Theta \xi \rangle 
+ i \langle \xi , \Theta g \eta \rangle 
+\frac{\tilde{\Omega}}{4} \coth(\tilde{\Omega}t)  
\langle \eta, \Theta g \Theta \eta \rangle \Big\}\;,
\\
\mathfrak{D}_\mu(\tfrac{i\partial}{\partial \xi},
\tfrac{i\partial}{\partial \eta}) \, \mathcal{E}
&=  \Big(
i \tilde{\Omega}^2 (\Theta g z)_\mu 
- \tilde{\Omega}\tanh(\tilde{\Omega}t)  z_\mu
+ i\frac{\tilde{\Omega}^2}{2} (\Theta g\Theta \xi)_\mu
\\
&\qquad
- \tilde{\Omega} \coth(\tilde{\Omega}t) (g\Theta \eta)_\mu  
+ \frac{\tilde{\Omega}}{2}\coth(\tilde{\Omega}t) (\Theta \eta)_\mu  
- \frac{\tilde{\Omega}}{2}\tanh(\tilde{\Omega}t) (\Theta \xi)_\mu 
\Big) \mathcal{E}\;.
\end{align*}
Then, the functions 
\begin{align*}
\mathcal{Z}_\mu &:=  \mathcal{E}^{-1} 
\mathfrak{D}_\mu(\tfrac{i\partial}{\partial \xi},
\tfrac{i\partial}{\partial \eta}) \, \mathcal{E}\Big|_{\xi=\eta=0}\;,
\\
\mathcal{N}_{\mu\nu} &:= 
\mathcal{Y}_{\mu\nu}+ \mathfrak{D}_\mu(\tfrac{i\partial}{\partial \xi},
\tfrac{i\partial}{\partial \eta}) \, \big(\mathcal{E}^{-1}
\mathfrak{D}_\nu(\tfrac{i\partial}{\partial \xi},
\tfrac{i\partial}{\partial \eta}) \, \mathcal{E}\big)\;,
\end{align*}
take the values given in the Lemma, and the assertion follows.
\end{proof}

A very nice feature of the results of \cite{CGRS1} is that 
both the questions of the Dixmier traceability and of the value
of the Dixmier trace of an operator of the form $aG^k$ are reduced 
to the value of the Hilbert-Schmidt norm of the heat-type 
operator $ae^{-tG^{-1}}$. In
our context $a=L_\star(f)$, $G=(1+\mathcal D_\bullet^2)^{-1}$, and all we
need to do is to evaluate the Hilbert-Schmidt norm of
$L_\star(f)e^{-t\D_\bullet^2}$.
\begin{Lemma}
\label{bounds0} 
If  $f\in\mathcal A_\star$, then we
have:
$$
\|L_\star(f)e^{-t\mathcal D_\bullet^2}\|_2^2
= \frac{\tilde{\Omega}^{d/2}}{
\pi^{d/2}  \tanh^{d/2}(2\tilde{\Omega} t)} 
\int d z\, \sqrt{\det g} \;\bar f\star f(z)
e^{-\tilde{\Omega} \tanh(2\tilde{\Omega}t) \langle z,g z\rangle}\;.
$$
\end{Lemma} 
\begin{proof}
Since $\D^2_\bullet=H\otimes 1-(-1)^\bullet\,\tilde\Omega\otimes \Sigma$, we
have
$$
0\leq e^{-t\mathcal D_\bullet^2}=e^{-tH}\otimes e^{(-1)^\bullet t \tilde\Omega\Sigma}\;,
$$
and thus
$$
\|L_\star(f)e^{-t\mathcal D_\bullet^2}\|_2=\|L_\star(f)e^{-tH}\|_2\,
{\rm tr}\big( e^{(-1)^\bullet 2t \tilde\Omega\Sigma}\big)^{1/2}\;.
$$
For the matrix trace, we have 
\begin{align*}
\mathrm{tr}(e^{(-1)^\bullet t\tilde{\Omega} \Sigma })
=  \mathrm{tr}(e^{(-1)^\bullet t\tilde{\Omega} \sum_{\mu=1}^d (b_\mu^* b_\mu - 
b_\mu b_\mu^*)})
=  \mathrm{tr}\Big( \prod_{\mu=1}^d e^{(-1)^\bullet t 
\tilde{\Omega} (b_\mu^* b_\mu - b_\mu b_\mu^*)}\Big)\;.
\end{align*}
In the basis $|s_1,\dots,s_d\rangle := 
(b_1^*)^{s_1}\cdots (b_d^*)^{s_d}|0,\cdots,0\rangle$ of
$\mathbb{C}^{2^d}$, with $s_i \in \{0,1\}$, we have 
$$
-(b_\mu^* b_\mu - 
b_\mu b_\mu^*)|s_1,\dots,s_d\rangle = (-1)^{s_\mu} 
|s_1,\dots,s_d\rangle\,,
$$ 
and therefore, for both $\bullet=1,2$, 
\begin{align*}
\mathrm{tr}(e^{(-1)^\bullet t \tilde{\Omega} \Sigma })
= 2^d\cosh^d (\tilde{\Omega}t)\;.
\end{align*}
The other bit, 
$\|L_\star(f)e^{-t H}\|_2^2
={\rm Tr}\big(e^{-tH}L_\star(\bar f\star f)e^{-tH}\big)
={\rm Tr}\big(L_\star(\bar f\star f)e^{-2tH}\big)$, has been computed
in Proposition~\ref{Proposition-trace}.
\end{proof}
\begin{Remark} 
{\rm Since for $f\in\A_\star$, $\bar f\star f$ 
is a priori {\it not a positive function},  in the  previous Lemma, 
one may wonder why $\int d z  \,\sqrt{\det g}\, \bar f\star f(z)
\exp\{-\tilde\Omega\tanh(2\tilde\Omega t) \langle z,gz\rangle\}$ 
is positive, as it should be. This follows from the following facts: 
For $A$ a positive definite matrix commuting with 
$\Theta$, set $g_A(x):=e^{-<x,Ax>}$. Then a computation gives 
$$
g_A\star g_A= (\det(1+\Theta^t A^2 \Theta))^{-1/2}\,g_B\;,
\quad
\mbox{with}\quad
B=\frac{2A}{1+\Theta^t A^2 \Theta}\;.
$$
It follows
$$
\exp\{-\tilde\Omega\tanh(2\tilde\Omega t) 
\langle z,g z\rangle\}=(\det(1+\Theta^t A^2 \Theta))^{1/2}\,g_A\star g_A\;,
$$
for 
$$
A=\frac{g^{-1}-\big(g^{-2}-\tilde\Omega^2\,\tanh(2\tilde\Omega t)^2
\Theta^t\Theta\big)^{1/2}}{
\tilde\Omega\tanh(2\tilde\Omega t)\,\Theta^t\Theta}\;.
$$
Note that $g^{-2}-\tilde\Omega^2\,\Theta^t\Theta
= (1-\tilde\Omega^2\,\Theta^t\Theta/4)^2$ so that $A$ exists for all $t$.
Using the traciality of the Moyal product \eqref{tracial}, we then get
for the matrix $A$ given above and up to a  positive constant:
\begin{align*}
\int d z \, \bar f\star f(z)
\exp\{-\tilde\Omega\tanh(2\tilde\Omega t) \langle z,g^{-1}z\rangle\}&=
C(\Theta,\tilde\Omega,t)\,\int d z \, \bar f\star f(z)\,g_A\star g_A(z)\\
&=
C(\Theta,\tilde\Omega,t)\,\int d z \, \bar f\star f \star g_A\star g_A(z)\\
&=
C(\Theta,\tilde\Omega,t)\,\int d z \, \overline{ f\star g_A}\star f \star g_A(z)\\
&=
C(\Theta,\tilde\Omega,t)\,\int d z \, \overline{ f\star g_A}(z)\, f \star g_A(z)\\
&=
C(\Theta,\tilde\Omega,t)\,\| f \star g_A\|_2\;\geq \;0\;.
\end{align*}
Moreover, it explains why $L_\star(f)e^{-t\mathcal D_\bullet^2}$ is
Hilbert-Schmidt also for $f$ in $\mathcal B_\star$, since $\A_\star$
is an ideal of $\mathcal B_\star$ and $\A_\star\subset L^2(\mathbb
R^d)$.}
\end{Remark}
\begin{Lemma}
\label{normalization}
For all $t>0$ and $\bullet=1,2$, we have
$$
{\rm Tr}\big(e^{-t\mathcal D_\bullet^2}\big)=\coth^d(\tilde\Omega t)\;.
$$
\end{Lemma}
\begin{proof}
This is a corollary of Lemma \ref{bounds0} and the Remark which
follows it, by letting $f$ going to the constant unit function.
\end{proof}

\begin{Lemma}
\label{bounds1}
If  $f\in\mathcal A_\star$ and  $t>0$, then we have the bound
$$
\|L_\star(f)e^{-t\mathcal D_\bullet^2}\|_2\leq C\,\|\bar f\star
f\|_1^{1/2}\, \max (1,t^{-d/4})\;,
$$
where the constant depends only on $\tilde{\Omega}$ and $\Theta$.
\end{Lemma} 
\begin{proof}
This is a direct consequence of Lemma \ref{bounds0}.
\end{proof}

\begin{Lemma}
\label{lemma-trucus}
There is $C'>0$ such that for any $f_1,f_2 \in \A_\star$ and $t>0$ one has
$$
\|L_\star(f_1) [L_\star(f_2),e^{-t \mathcal{D}_\bullet^2}]\|_1 
\leq C' t^{1/2} \sum_{\mu=1}^d \|f_1\|_2 \, \| L_\star(\partial_\mu f)
e^{-t \mathcal{D}_\bullet^2/4}]\|_1\;.
$$
\end{Lemma}
\begin{proof}
By Lemma \ref{normalization}, $e^{-t\mathcal{D}_\bullet^2}$ is trace
class for $t>0$. We use the identity
\begin{align}
\label{commm}
[e^{A},B]
= \int_0^1 ds \frac{d}{ds} 
\big(e^{sA}Be^{(1-s)A}\big)
= \int_0^1 ds \;
e^{sA}[A,B]e^{(1-s)A}\;,
\end{align}
to get
$$
L_\star(f_1) [L_\star(f_2),e^{-t\mathcal D_\bullet^2}]
= -tL_\star(f_1)  \int_0^1 ds \;
e^{-ts\mathcal{D}^2_\bullet}[\mathcal{D}^2_\bullet,L_\star(f_2)]
e^{-t(1-s)\mathcal{D}^2_\bullet}\;.
$$
Hence we have
\begin{align*}
\|L_\star(f_1) [L_\star(f_2),e^{-t\mathcal D_\bullet^2}]\|_1
&\leq 
t \|L_\star(f_1)\| \int_0^1 ds \;\Big(
\|e^{-ts\mathcal{D}^2_\bullet/2}\D_\bullet\|\,
\|e^{-ts\mathcal{D}^2_\bullet/2}[\mathcal{D}_\bullet,
L_\star(f_2)]e^{-t(1-s)\mathcal{D}^2_\bullet}\|_1
\\
&\qquad\qquad\qquad
+ \|e^{-ts\mathcal{D}^2_\bullet}[\mathcal{D}_\bullet,L_\star(f_2)]
e^{-t(1-s)\mathcal{D}^2_\bullet/2}\|_1\|\D_\bullet
e^{-t(1-s)\mathcal{D}^2_\bullet/2}\|\Big)\;.
\end{align*}
By spectral theory, $\|\D_\bullet
e^{-t\mathcal{D}^2_\bullet}\|=(2et)^{-1/2}$. Thus, using the relation
$$
[\mathcal{D}_\bullet, L_\star(f_2) ]=\begin{cases}
i L_\star(\partial_\mu f_2) \otimes \Gamma^\mu\,,\quad\bullet=1\;,
\\
i L_\star(\partial_\mu f_2) \otimes \Gamma^{\mu+4}\,,\quad\bullet=2\;,
\end{cases}
$$
we get with $\|L_\star(f_1)\|\leq C \|f_1\|_2$ and $C':= C \sqrt{2}
\pi e^{-1} \sup_{\mu=1}^{2d}\|\Gamma^\mu \|$
\begin{align*}
&\|L_\star(f_1) [L_\star(f_2),e^{-t\mathcal D_\bullet^2}]\|_1
\\
&\qquad\leq \frac{C'}{\pi} t^{1/2} \|f_1\|_2 \sum_{\mu=1}^d\int_0^1 ds \,
s^{-1/2}(1-s)^{-1/2}\|e^{-ts\D^2_\bullet/2}L_\star(\partial_\mu f_2)
e^{-t(1-s)\D^2_\bullet/2}\|_1\;.
\end{align*}
Estimating
$$
\|e^{-ts\D^2_\bullet/2}L_\star(\partial_\mu f_2)e^{-t(1-s)\D^2_\bullet/2}\|_1\leq
\begin{cases}
\|L_\star(\partial_\mu f_2)e^{-t\D^2_\bullet/4}\|_1,\quad&
\mbox{if}\quad s\in[0,1/2]\;,\\
\|e^{-t\D^2_\bullet/4}L_\star(\partial_\mu f_2)\|_1,\quad&
\mbox{if}\quad s\in[1/2,1]\;,
\end{cases}
$$
the result follows.
\end{proof}

\begin{Lemma}
\label{bounds2}
Let $f\in\A_\star$. Then, there exists a finite constant $C(f)$ such
that for all $t>0$:
$$
\|L_\star(f) e^{-t\mathcal D_\bullet^2}\|_1\leq C(f) \,\max(t^{-d/2},t^{d/2})\;.
$$
\end{Lemma}
\begin{proof}
Our strategy is to iterate a combination of the factorization 
property \eqref{factor} with Lemma \ref{bounds1} and 
Lemma \ref{lemma-trucus} far enough so that we can bound
$e^{-\epsilon t\mathcal
 D_\bullet^2}$ alone in trace-norm (i.e.  without element of the
algebra of both sides). 

According to \eqref{factor}, for all $f \in
\A_\star$ there exist $f_1,f_2 \in \A_\star$ such that $f = f_1
\star f_2$, giving 
$$
L_\star(f) e^{-t\mathcal D_\bullet^2}
= L_\star(f_1) e^{-t\mathcal D_\bullet^2}L_\star(f_2)
+L_\star(f_1) \big[L_\star(f_2),e^{-t\mathcal D_\bullet^2}\big]\;.
$$
From Lemma \ref{bounds1} and Lemma \ref{lemma-trucus} we conclude
\begin{align}
\label{trucus}
\|L_\star(f) e^{-t\mathcal D_\bullet^2}\|_1
&\leq  \|L_\star(f_1) e^{-t\mathcal D_\bullet^2/2}\|_2 \,
\|e^{-t\mathcal{D}_\bullet^2/2} L_\star(f_2)\|_2
+ \|L_\star(f_1) \big[L_\star(f_2),e^{-t\mathcal D_\bullet^2}\big]\|_2
\\
&\leq  \| \bar{f}_1 \star f_1\|_1^{1/2}
\| \bar{f}_2 \star f_2\|_1^{1/2} \max(t^{-d/2},1)
+ C' t^{1/2} \sum_{\mu=1}^d \|f_1\|_2 \,\|L_\star(\partial_\mu f_2)
e^{-t\mathcal{D}^2_\bullet/4}\|_1\;.\nonumber
\end{align}
Iterating $d$-times the estimate \eqref{trucus} with the
repeated factorization
$$
\partial_{\mu_1} f_2=f_{1,\mu}\star f_{2,\mu}\;,\quad \dots\;,\quad 
\partial_{\mu_{k+1}} f_{2,\mu_1\dots\mu_k}=f_{1,\mu_1\dots\mu_{k+1}}
\star f_{2,\mu_1\dots\mu_{k+1}}\;,
$$
with $f_{1,\mu}, f_{2,\mu},\dots,f_{1,\mu_1\dots\mu_{k+1}}
\star f_{2,\mu_1\dots\mu_{k+1}} \in\A_\star$,
we get for some constants $C_0(f),\dots,C_d(f)$ depending on $f$ and
on the choice of factorization at each step: 
\begin{align*}
\|L_\star(\partial_\mu f) e^{-t\mathcal D_\bullet^2}\|_1& \leq
\sum_{k=0}^{d-1} C_k(f) t^{k/2} \max(t^{-d/2},1) 
+ C_d(f) t^{d/2}
\sum_{\mu_1,\dots,\mu_d=1}^d\|L_\star(f_{2,\mu_1\dots\mu_d}) 
e^{-t\D^2_\bullet/4^d}\|_1\;.
\end{align*}
Using Lemma \ref{normalization} we get
$$
\|L_\star(f_{2,\mu_1\dots,\mu_d})e^{-t\D^2_\bullet/4^d}\|_1\leq
\|L_\star(f_{2,\mu_1\dots,\mu_d})\|\,\|e^{-t\D^2_\bullet/4^d}\|_1\leq
C''\|f_{2,\mu_1,\dots,\mu_d}\|_2 \max(t^{-d},1)\;,
$$
which completes the proof.
\end{proof}
\begin{Corollary}
\label{cor-traceclass}
For any $f\in \mathcal{A}_\star$, the operator 
$[(\mathcal{D}_\bullet^2+1)^{-d/2},L_\star(f)]$ is of
trace class.
\end{Corollary}
\begin{proof}
By factorization $f=f_1\star f_2$, $f_1,f_2\in\A_\star$ and  Leibniz rule:
$$
[(\mathcal{D}_\bullet^2+1)^{-d/2},L_\star(f)]=L_\star(f_1)[(\mathcal{D}_\bullet^2+1)^{-d/2},L_\star(f_2)]-\big(L_\star(\overline f_2)[(\mathcal{D}_\bullet^2+1)^{-d/2},L_\star(\overline f_1)]\big)^*\;,
$$
it suffices to show that $L_\star(f_1)[(\mathcal{D}_\bullet^2+1)^{-d/2},L_\star(f_2)]$ is of trace class for arbitrary $f_1,f_2\in\A_\star$.
By spectral theory,
$$
L_\star(f_1)[(\mathcal{D}_\bullet^2+1)^{-d/2},L_\star(f_2)]
=\frac{1}{\Gamma(d/2)} \int_0^\infty dt \;t^{d/2-1}
L_\star(f_1)[e^{- t(\mathcal{D}_\bullet^2+1)} ,L_\star(f_2)]\;.
$$
Combining Lemma \ref{lemma-trucus} with Lemma \ref{bounds2} we obtain for a finite constant
depending only on $f$:
\begin{align*}
&\|L_\star(f_1)[(\mathcal{D}_\bullet^2+1)^{-d/2},L_\star(f_2)]\|_1
\leq  C(f)\,\int_0^\infty dt \;
e^{-t} \,t^{d/2-1}\cdot t^{1/2}\cdot \max(t^{-d/2},t^{d/2})\;.
\end{align*}
As the integral converges, we are done.
\end{proof}

We have arrived at the main result of this subsection, that the
spectral triple $(\mathcal{A}_\star,\mathcal{H},\mathcal{D}_\bullet)$ has
metric dimension $d$ and not $2d$ (remember that $d$ is even).
\begin{Theorem}
 For $f\in\A_\star$, $\bullet=1,2$, the operator $L_\star(f)
 (1+\mathcal D_\bullet^2)^{-d/2}$ belongs to $\mathcal
 L^{1,\infty}(\mathcal H)$ and for any Dixmier trace ${\rm
   Tr}_\omega$, we have
$$
{\rm Tr}_\omega\big(L_\star(f) (1+\mathcal D_\bullet^2)^{-d/2}\big)
=\frac{1}{\pi^{d/2} (d/2)!} \int dx\,\sqrt{\det g}\;f(x)\;.
$$
\end{Theorem}
\begin{proof}
We use  the factorization property \eqref{factor} to write
$f=f_1\star f_2$ with $f_1,f_2\in\A_\star$, which gives
\begin{align*}
L_\star( f) (1+\mathcal D_\bullet^2)^{-d/2}
=L_\star( f_1) (1+\mathcal D_\bullet^2)^{-d/2}L_\star( f_2)+
L_\star( f_1) \big[L_\star( f_2), (1+\mathcal D_\bullet^2)^{-d/2}\big] \;.
\end{align*}
By Corollary \ref{cor-traceclass}, $L_\star( f_1) [L_\star( f_2),
(1+\mathcal D_\bullet^2)^{-d/2}]$ is trace class, and combining Lemma
\ref{bounds0} and \cite[Proposition 4.8]{CGRS1} we get that $L_\star(
f_1) (1+\mathcal D_\bullet^2)^{-d/2}L_\star(f_2)$ belongs to $\mathcal
L^{1,\infty}(\mathcal H)$. Therefore, $L_\star( f) (1+\mathcal
D_\bullet^2)^{-d/2}$ is Dixmier-trace-class too, and any of its
Dixmier trace coincides with those of $L_\star( f_1) (1+\mathcal
D_\bullet^2)^{-d/2}L_\star( f_2)$. 

Using a polarization identity, it
suffices to compute 
$\mathrm{Tr}_\omega \big(L_\star( \bar{f}) (1+\mathcal
D_\bullet^2)^{-d/2}L_\star( f)\big)$. Note that
\begin{align*}
&\lim_{s\to 1^+}(s-1){\rm Tr}\big(L_\star(\bar f) 
(1+\mathcal D_\bullet^2)^{-ds/2}L_\star(f)\big)
\\
&\quad=\lim_{s\to 1^+}\frac{s-1}{\Gamma(ds/2)}
\int_0^\infty dt \,e^{-t}\, t^{ds/2-1}
{\rm Tr}\big(L_\star(\bar f) e^{-t\mathcal D_\bullet^2}L_\star(f)\big)
\\
&\quad=\lim_{s\to 1^+}\frac{s-1}{\Gamma(ds/2)}
\int_0^\infty dt \,e^{-t}\, t^{d(s-1)/2-1} 
\frac{(\tilde{\Omega} t)^{d/2}}{
\pi^{d/2} \tanh^{d/2}(\tilde{\Omega} t)} 
\int d z  \, \sqrt{\det g}\;\bar f\star f(z)
e^{-\tilde{\Omega} \tanh(\tilde{\Omega}t) \langle z,g z\rangle}
\\
&\quad=\frac{1}{\pi^{d/2}} \lim_{s\to 1^+}\frac{(s-1)
\Gamma(d(s-1)/2)}{\Gamma(ds/2)}\int d z  \,\sqrt{\det g}\;\bar f\star f(z)\;.
\end{align*}
Now \cite[Proposition 5.13]{CGRS1}, which relies on 
Corollary \ref{cor-traceclass}, gives for any Dixmier trace
\begin{align}
\label{azerty}
{\rm Tr}_\omega\big(L_\star(\bar f) (1+\mathcal D_\bullet^2)^{-d/2}
L_\star(f)\big)
=\frac{1}{\pi^{d/2}(d/2)!} \int dx\,\sqrt{\det g}\;\bar f\star f(x)\;.
\end{align}
This all what we needed to
prove.
\end{proof}

\subsection{Regularity and dimension spectrum}

Our next task is to check the regularity condition.

\begin{Proposition}
\label{Proposition:regularity}
For any $f \in \mathcal{B}_\star$ and $\bullet=1,2$, both $L_\star(f)$
and $[\mathcal{D}_\bullet,L_\star(f)]$ belong to $\bigcap_{n=1}^\infty
\mathrm{dom}\;\delta_\bullet^n$, where $\delta_\bullet
(T):=[\langle\mathcal{D}_\bullet\rangle,T]$ and $\langle
\mathcal{D}_\bullet\rangle:= (\mathcal{D}_\bullet^2+1)^{\frac{1}{2}}$.
\end{Proposition}
\begin{proof}
 It is well known  (see for example \cite{CGRS2}) that $\bigcap_{n=1}^\infty
 \mathrm{dom}\;\delta_\bullet^n=\bigcap_{n,m=1}^\infty
 \mathrm{dom}\;R_\bullet^m\circ L_\bullet^n$ where $R_\bullet$ and
 $L_\bullet$ are the unbounded linear operators given by
$$
R_\bullet(T):=[\mathcal D_\bullet^2,T]\langle 
\mathcal D_\bullet\rangle^{-1},\quad 
L_\bullet(T):=\langle\mathcal D_\bullet\rangle^{-1}
[\mathcal D_\bullet^2,T]\;.
$$
By an easy inductive argument, we see that for 
$T\in\mathcal B_\star\cup[\mathcal D_\bullet,\mathcal B_\star]$, we have
$$
R_\bullet^m\circ L_\bullet^n(T)=\langle\mathcal D_\bullet\rangle^{-n}\,\big({\rm ad}(\mathcal D_\bullet^2)\big)^{n+m}(T)\,\langle\mathcal D_\bullet\rangle^{-m}\;.
$$
Since $\mathcal{D}_\bullet^2=H-(-1)^\bullet\tilde{\Omega}\Sigma$ with $\Sigma$ bounded and 
$[H,\Sigma]=0$, we get
\begin{align}
\label{RL}
R_\bullet^m\circ L_\bullet^n(T)=\sum_{k=0}^{n+m} \binom{n+m}{k} 
\big((-1)^{\bullet+1}\tilde{\Omega} \,
\mathrm{ad}(\Sigma)\big)^{n+m-k}
\bigg(\langle\mathcal D_\bullet\rangle^{-n}\,
\big({\rm ad}(H)\big)^{k}(T)\,\langle\mathcal D_\bullet\rangle^{-m}\bigg)\;.
\end{align}
We treat the worst case only, which is when $k=n+m$ (for the other
values of $k$, one may use similar but simpler arguments). That is, we
need to show the boundedness of
$$
\langle\mathcal D_\bullet\rangle^{-n}\,
\big({\rm ad}(H)\big)^{n+m}(T)\,\langle\mathcal D_\bullet\rangle^{-m},
\quad n,m\in\mathbb N\,.
$$
The preceding expression applied to 
$T=[\mathcal{D}_1,L_\star(f)]=
i \Gamma^\mu  L_\star(\partial_\mu f)$ 
gives 
\begin{align*}
&\langle\mathcal D_1\rangle^{-n}\,\Gamma^\mu \big({\rm
ad}(H)\big)^{n+m}(L_\star(i \partial_\mu f))\,
\langle\mathcal D_1\rangle^{-m}
\\
&\qquad= 
\big(\Gamma^\mu 
+\langle\mathcal D_1\rangle^{-n} [\Gamma^\mu, \langle\mathcal
D_1\rangle^{n} ] \big)\langle\mathcal D_1\rangle^{-n} 
\big({\rm
ad}(H)\big)^{n+m}(L_\star(i \partial_\mu f))\,
\langle\mathcal D_1\rangle^{-m}\;,
\end{align*}
and using $[\Gamma^\mu, \langle\mathcal D_1\rangle^{n} ]=[\Gamma^\mu,
(\tilde{\Omega}\Sigma)^n]$ it is enough to treat the case
$T\in\mathcal B_\star$.  Similarly for $\mathcal{D}_2$.  Now, Lemma
\ref{nabla-com} shows that $(\mathrm{ad}(H))^{n+m} (T)$ can be written
as a sum of terms of the form $\hat{\nabla}^k
L_\star(f)\hat{\nabla}^l$ with $f\in\mathcal B_\star$ and $k$ (resp.
$l$) not exceeding $n$ (resp.  $m$), where $\hat{\nabla}$ is $\nabla$
or $\tilde{\nabla}$. Then, one concludes using Corollary
\ref{nabla-com2}.
\end{proof}

In the next and in analogy with the regularity condition, we will
prove that one can determine the dimension spectrum with the
derivations $R_\bullet$ and $L_\bullet$, instead of $\delta_\bullet$.
\begin{Proposition}
Let $b $ belonging to the polynomial algebra generated by
$\delta_\bullet^n(\A_\star)$ and 
$\delta_\bullet^n([\mathcal D_\bullet,\A_\star])$. Let
also $\zeta_b(z):={\rm Tr}(b\langle \mathcal{D}_\bullet\rangle^{-z})$,
defined on the open half plane $\Re(z)>2d$. Then, for all
$M\in\mathbb R$, $\zeta_b$ is a finite sum of terms of the form
${\rm Tr}\big(R_\bullet^{n_1}(b_1)\cdots
R_\bullet^{n_k}(b_k)\langle\mathcal{D}_\bullet\rangle^{-z-m}\big)$,
$n_j,k,m\in\mathbb N$, $b_j\in \mathcal A_\star\cup [\mathcal
D_\bullet,\mathcal A_\star]$, plus a function holomorphic on the half plane
$\Re(z)>M$.
\label{Prop-delta-R}
\end{Proposition}
\begin{proof}
For $\Re(z)>2d$, $\langle \mathcal{D}_\bullet\rangle^{-z}$ is a trace
class operator. Since the algebra generated by $\delta_\bullet^n( L_\star
(f))$ and $\delta_\bullet^n( [\mathcal{D}_\bullet,L_\star (f)])$ consists by
Proposition~\ref{Proposition:regularity} of \emph{bounded
  operators}, $b \langle \mathcal D_\bullet\rangle^{-z}$ is trace class for
$\Re(z)>2d$ and any $b$ in this polynomial algebra.

So let $b\in\mathcal A_\star\cup [\mathcal D_\bullet, \mathcal A_\star]$. 
From the spectral representation of a positive operator
$A$:
$$
A=\frac{1}{\pi} \int_0^\infty d\lambda\,
\lambda^{-1/2}\frac{A^2}{A^2+\lambda}\;,
$$
we get
$$
\delta_\bullet(b)=\frac{1}{\pi}
\int_0^\infty d\lambda\,\lambda^{1/2}\frac1{\langle
\mathcal{D}_\bullet\rangle^2+\lambda}[\langle
\mathcal{D}_\bullet\rangle^2, b]\frac1{\langle
\mathcal{D}_\bullet\rangle^2+\lambda}\;.
$$
Commuting $[\langle\mathcal{D}_\bullet\rangle^2, b]$ with $(\langle
\mathcal{D}_\bullet\rangle^2+\lambda)^{-1}$ to the left, we get after some
re-arrangements and using $\int
d\lambda\,\lambda^{1/2}t(t^2+\lambda)^{-2}=\pi/2$:
$$
\delta_\bullet(b)=\tfrac12 R_\bullet(b)-\frac1\pi\int_0^\infty 
d\lambda\,\lambda^{1/2}\frac1{\langle
\mathcal{D}_\bullet\rangle^2+\lambda}R_\bullet^2(b)\frac{\langle
\mathcal{D}_\bullet\rangle^2}{(\langle
\mathcal{D}_\bullet\rangle^2+\lambda)^2}\;.
$$
This suggests to introduce the map $\mathcal T_\bullet:\mathcal B(\mathcal H)\to \mathcal B(\mathcal H)$ given by
$$
A\mapsto \mathcal T_\bullet(A):=\frac2\pi\int_0^\infty d\lambda\,\lambda^{1/2}\frac1{\langle
\mathcal{D}_\bullet\rangle^2+\lambda}A\frac{\langle
\mathcal{D}_\bullet\rangle^2}{(\langle
\mathcal{D}_\bullet\rangle^2+\lambda)^2}\,.
$$
Note that this operator is contractive. Indeed
$$
\|\mathcal T_\bullet\|\leq \frac2\pi\int_0^\infty \frac{\lambda^{1/2}}{(1+\lambda)^2}\,
d\lambda=1\;.
$$
Thus $2\delta_\bullet=R_\bullet-\mathcal T_\bullet\circ R_\bullet^2$ and since $\mathcal T_\bullet$ commutes with $R_\bullet$ (because
$R_\bullet$ commutes with the operators of left and right multiplications by
functions of $\langle\mathcal{D}_\bullet\rangle$), we get
$$
2^n\delta_\bullet^n=\sum_{k=0}^n{n\choose k}\,(-1)^k \mathcal T_\bullet^k\circ R_\bullet^{n+k}\;.
$$
Hence, a typical element of the algebra generated by $\delta_\bullet^n( L_\star (f))$ and
$\delta_\bullet^n( [\mathcal{D}_\bullet,L_\star (f)])$,  is a finite sum of
elements of the form
$$
\prod_{j=1}^k\mathcal T_\bullet^{n_j}\circ R_\bullet^{m_j}(b_j), \quad 
b_j\in\mathcal A_\star\cup [\mathcal D_\bullet, \mathcal A_\star],\quad n_j<m_j\in\mathbb N\;.
$$
From the same reasoning as at the beginning of the proof, the function
$$
\zeta_{R,\mathcal T}(b_1,n_1,m_1;\cdots;b_k,n_k,m_k;z)
:={\rm Tr}\Big(\prod_{j=1}^k\mathcal T_\bullet^{n_j}\circ R_\bullet^{m_j}(b_j)
\langle\mathcal{D}_\bullet\rangle^{-z}\Big)\;,
$$
is holomorphic on the open half plane $\Re(z)>2d$. Starting from the
definition, we have
$$
\prod_{j=1}^k\mathcal T_\bullet^{n_j}\circ R_\bullet^{m_j}(b_j)=\int_{[0,\infty]^{|n|}}\!\!\!
d\lambda^{|n|} \prod_{j=1}^k\Big(\prod_{r_j=1}^{n_j}
\frac{(2/\pi)\lambda_{r_j}^{1/2}}{\langle
\mathcal{D}_\bullet\rangle^2+\lambda_{r_j}}\Big)\,
R_\bullet^{m_j}(b_j)\,\Big(\prod_{s_j=1}^{n_j}
\frac{\langle\mathcal{D}_\bullet\rangle^2}{(\langle
\mathcal{D}_\bullet\rangle^2+\lambda_{s_j})^2}\Big)\;.
$$
The next step consists in commuting for each $j$ the $R_\bullet^{m_j}(b_j)$ to
the left of $(\langle\mathcal{D}_\bullet\rangle^2+\lambda_{r_j})^{-1}$:
\begin{align}
\label{DR}
\Big[\frac1{\langle\mathcal{D}_\bullet\rangle^2+\lambda_j}, R_\bullet^{m_j}(b_j)\Big]
& =-\frac1{\langle\mathcal{D}_\bullet\rangle^2+\lambda_j}
\,R_\bullet^{m_j+1}(b_j)\,\frac{\langle\mathcal{D}_\bullet\rangle}{
\langle\mathcal{D}_\bullet\rangle^2+\lambda_j}
\nonumber
\\
&= \sum_{p_j=1}^{N_j} (-1)^{p_j}
\,R_\bullet^{m_j+p_j}(b_j)\,\frac{\langle\mathcal{D}_\bullet\rangle^{p_j}}{
(\langle\mathcal{D}_\bullet\rangle^2+\lambda_j)^{p_j+1}}
\nonumber
\\
&\quad + \frac{(-1)^{N_j+1}}{\langle\mathcal{D}_\bullet\rangle^2+\lambda_j}
\,R_\bullet^{m_j+N_j+1}(b_j)\,\Big(\frac{\langle\mathcal{D}_\bullet\rangle}{
\langle\mathcal{D}_\bullet\rangle^2+\lambda_j}\Big)^{N_j+1}\;.
\end{align}
Any of the resulting $\lambda$-integrals is convergent, and $R_\bullet^n(b)$ 
is bounded for all $n\in\mathbb N$.
Choosing the $N_j$ large enough, we generate as much negative 
powers of $\langle\mathcal{D}_\bullet\rangle$ as necessary to make the
product of the remainder with $\langle\mathcal{D}_\bullet\rangle^{-z}$ a
trace-class operator for any given $z$ with $\Re(z)>M$ (if 
$M$ gets more and more negative we need larger and larger $N_j$). 
The other terms integrate to 
\[
\int_{[0,\infty]} d\lambda_j \frac{(2/\pi)\lambda_j^{1/2}
\langle \mathcal D_\bullet\rangle^{p_j+2}}{(\langle
\mathcal{D}_\bullet\rangle^2+\lambda_{r_j})^{p_j+3} }
= \frac{\Gamma(\frac{3}{2}+p_j)\langle \mathcal D_\bullet\rangle^{-1-p_j}}{\sqrt{\pi}
\Gamma(3+p_j)} \;,
\]
so that, up to the remainder term,  which is easily seen to be holomorphic on the open half
plane $\Re(z)>M$, $\prod_{j=1}^k\mathcal  T_\bullet^{n_j} R_\bullet^{m_j}(b_j)$ is a finite 
linear combination of
\[
R_\bullet^{m_1+q_1}(b_1) \langle \mathcal D_\bullet\rangle^{-n_1-q_1}
R_\bullet^{m_2+q_2}(b_2) \langle \mathcal D_\bullet\rangle^{-n_2-q_2}
\cdots 
R_\bullet^{m_k+q_k}(b_k) \langle \mathcal D_\bullet\rangle^{-n_k-q_k}
\;.
\]
The final step consists in commuting the 
$\langle\mathcal  D_\bullet\rangle^{-n_j-q_j}$ to the right. If 
$n_j+q_j=2l_j$ is even, we use the ($\lambda_j=0$)-case of (\ref{DR}).
If $n_j+q_j=2l_j-1$ is odd, 
$$
\big[\langle \mathcal D_\bullet\rangle^{1-2l_j},R_\bullet^{m'_{j+1}}(b_{j+1})\big]
= \langle \mathcal D_\bullet\rangle^{-2l_j} \delta_\bullet(R_\bullet^{m'_{j+1}}(b_{j+1})) 
+\big[\langle \mathcal D_\bullet\rangle^{-2l_j},R_\bullet^{m'_{j+1}}(b_{j+1})\big]
\langle \mathcal D_\bullet\rangle\;.
$$
Using $\delta_\bullet=\frac{1}{2}(R_\bullet-\mathcal T_\bullet\circ R_\bullet^2)$, this case is reduced to the
first one. Eventually, we conclude that, up to a remainder term, which again  is easily seen to be
holomorphic on the open half plane $\Re(z)>M$, $\prod_{j=1}^k \mathcal T_\bullet^{n_j}\circ
R_\bullet^{m_j}(b_j)$ is a finite linear combination of
$$
R_\bullet^{m_1'}(b_1)R^{m_2'}(b_2)\cdots R^{m_k}(b_k)\langle
\mathcal D_\bullet\rangle^{-m}
\;.
$$
This concludes the proof.
\end{proof}

We can now state the main result of this section, namely:

\begin{Theorem}
 For $\bullet=1,2$, the spectral triple
 $(\mathcal{A}_\star,\mathcal{H},\mathcal{D}_\bullet)$ has dimension
 spectrum $\mathrm{Sd}=d-\mathbb{N}$. Moreover, all poles of
 $\zeta_b(z)$ at $z \in \mathrm{Sd}$ are simple with local residues,
 i.e.\ for $b=\delta_\bullet^{n_1}L_\star(f_1) \cdots
 \delta_\bullet^{n_k}L_\star(f_k)$, any residue $\mathrm{res}_{z \in
   \mathrm{Sd}} \zeta_b(z)$ is a finite sum of terms of the form
$$
\displaystyle \int_{\mathbb{R}^d} dx\;x^{\alpha_0} \star 
(\partial^{\alpha_1}f_1) \star \cdots \star 
(\partial^{\alpha_k}f_k)\;,
$$
where $\alpha_i\in\mathbb N^d$. 
An analogous result holds when the $L_\star(f_k)$'s in $b$ are
replaced by $[\mathcal{D}_\bullet,L_\star(f_k)]$'s. 
\end{Theorem}
\begin{proof}
According to Proposition~\ref{Prop-delta-R}, it is equivalent to consider the 
functions 
$$
{\rm Tr}\big(R_\bullet^{m_1}(b_1)\cdots R_\bullet^{m_k}(b_k) 
\langle\mathcal{D}_\bullet\rangle^{-z}\big),\quad 
b_i\in\mathcal A_\star\cup[\mathcal D_\bullet,\mathcal A_\star]\;,
$$
instead of $\zeta_b(z)$. These functions 
are well  defined for $\Re(z)>2d$, since $R_\bullet^m(b)$ is bounded.

Since  $[H,\Sigma]=0$, we get from \eqref{RL}
$$
R_\bullet^{m_j}(b_j)=\sum_{k=0}^{m_j} \binom{m_j}{k} 
\big({\rm ad}(H)\big)^{k}
\bigg(\big(\mathrm{ad}(\tilde{\Omega}\Sigma)\big)^{m_j-k}(b_j)\bigg)\,
\langle\mathcal D_\bullet\rangle^{-m_j}\;.
$$
Since $b_j$ is either $L_\star(f_j)$ or  $i\Gamma^\mu
L_\star(\partial_\mu f_j)$ for $\bullet=1$ and 
$i\Gamma^{\mu+d}
L_\star(\partial_\mu f_j)$  for $\bullet=2$, we get
$$
\mathrm{ad}(\tilde{\Omega}\Sigma)\big)^{m_j-k}(b_j)=\begin{cases}
\delta_{k,m_j} L_\star(f_j)\\
\mathrm{ad}(\tilde{\Omega}\Sigma)\big)^{m_j-k}i\Gamma^\mu L_\star(\partial_\mu f_j)
\\
\mathrm{ad}(\tilde{\Omega}\Sigma)\big)^{m_j-k}i\Gamma^{\mu+d} L_\star(\partial_\mu f_j)\;.
\end{cases}
$$ 
We may therefore assume that $b_j=M_jL_\star(f_j)$ with
$f_j\in\mathcal A$ and $M_j\in{\rm Mat}_{2^d}(\mathbb C)$. Thus, the
product $R_\bullet^{m_1}(b_1)\cdots R_\bullet^{m_k}(b_k)$ can be expressed as a finite
sum of terms of the form
$$
\big({\rm ad}(H)\big)^{n_1}\big(L_\star(f_1')\big)\,M_1\,
\langle\mathcal D_\bullet\rangle^{-m_1}
\cdots \big({\rm ad}(H)\big)^{n_k}\big(L_\star(f'_k)\big)\,M_k\,
\langle\mathcal D_\bullet\rangle^{-m_k},\quad n_j\leq m_j\;.
$$
Using the table given in Lemma \ref{nabla-com} we can express
$R_\bullet^{m_1}(b_1)\cdots R_\bullet^{m_k}(b_k)\langle\mathcal
D_\bullet\rangle^{-z}$ as a finite sum of terms
\begin{align*}
&L_\star(\partial^{\alpha_1} f_1)\,M_1\,
P_{\alpha_1}(\hat{\nabla})\langle\mathcal D_\bullet\rangle^{-m_1}
\cdots 
L_\star(\partial^{\alpha_k} f_k)\,M_k\,
P_{\alpha_k}(\hat{\nabla})\langle\mathcal D_\bullet\rangle^{-m_k}
\langle\mathcal D_\bullet\rangle^{-z}
\\
&\quad = \frac{1}{\Gamma(\frac{m_1}{2})
\cdots \Gamma(\frac{m_{k-1}}{2})\Gamma(\frac{m_{k}+z}{2})}
\int dt_1\cdots dt_k \;t_1^{\frac{m_1}{2}-1}\cdots 
t_k^{\frac{m_k+z}{2}-1}\; 
\\
&\qquad\qquad\times L_\star(\partial^{\alpha_1} f_1)\,M_1\,
P_{\alpha_1}(\hat{\nabla}) \, e^{-t_1 (\mathcal{D}_\bullet^2 +1)}\cdots 
L_\star(\partial^{\alpha_k} f_k) \,M_k\,P_{\alpha_k}(\hat{\nabla})\,
e^{-t_k (\mathcal{D}_\bullet^2+1)}\;,
\end{align*}
where $P_{\alpha_j}(\hat{\nabla})$ is a polynomial in 
$\nabla,\tilde{\nabla}$ of degree $|\alpha_j|\leq m_j$
(the $\alpha_j$ are multi-indices).

Using 
$$
[e^{-t_j(\mathcal{D}_\bullet^2+1)},T]=-t_j \int_0^1 ds_j
e^{-t_js_j(\mathcal{D}_\bullet^2+1)}[\mathcal{D}_\bullet^2,T]
e^{-t_j(1-s_j)(\mathcal{D}_\bullet^2  +1)}\;,
$$
we commute all heat operators $e^{-t_j(\mathcal{D}_\bullet^2+1)}$ to
the right, producing in each step a factor of $t_j$. The commutators
$[H,T]$ are expressed by Lemma \ref{nabla-com} and produce in each step at
most one derivative $\nabla$. In the terms with all heat operators
already on the right we then commute the derivatives $\hat{\nabla}$ to the
right of all functions $f_j$ but left of all heat operators. The
result is a finite sum of terms (with redefined $f_j,M_j$)
\begin{align}
\label{residue}
X&=\frac{\int_{[0,1]^N} ds \;P(s)}{\Gamma(\frac{m_1}{2})
\cdots \Gamma(\frac{m_{k-1}}{2})\Gamma(\frac{m_{k}+z}{2})}
\int dt_1\cdots dt_k \;t_1^{\frac{m_1}{2}+\beta_1-1}\cdots 
t_{k-1}^{\frac{m_{k-1}}{2}+\beta_{k-1}-1}
t_k^{\frac{m_k+z}{2}+\beta_k-1}\;
e^{-t} 
\nonumber
\\
&
\times 
\big(M_1 \cdots M_k e^{(-)^\bullet\tilde{\Omega}(t_1+\dots+t_k)\Sigma}\big)\,
L_\star\big(\partial^{\gamma_1} f_1 \star \cdots \star 
\partial^{\gamma_k} f_k\big) 
P_{\gamma_1,\dots,\gamma_k }(\hat{\nabla}) \, e^{-(t_1+\dots+t_k) H}\;,
\end{align}
with $|m+\beta|\geq |\gamma|$ and $|\beta|\leq K$ ($K$ can be chosen
as big as one wishes by pushing the expansion far enough), plus a
finite sum of remainders (with redefined $f_j,M_j$)
\begin{align}
\label{remainder}
Y&=\frac{1}{\Gamma(\frac{m_1}{2})
\cdots \Gamma(\frac{m_{k-1}}{2})\Gamma(\frac{m_{k}+z}{2})}
\int dt_1\cdots dt_k \;t_1^{\frac{m_1}{2}+\beta_1'-1}\cdots 
t_{k-1}^{\frac{m_{k-1}}{2}+\beta'_{k-1}-1} t_k^{\frac{m_k+z}{2}+\beta'_k-1}\; 
\nonumber
\\*
&\quad
\times 
\int_{[0,1]^{N'}} ds\; P'(s) \prod_{j=1}^k 
L_\star(\partial^{\gamma_j'} f_j) M_j 
P_{\gamma_j'}(\hat{\nabla}) e^{-\tau_j (\mathcal{D}_\bullet^2 +1)}\;,
\end{align}
with $|\gamma'| \leq |m+\beta'|$ and $|\beta'|>K$, where 
$\tau_j$ are positive functions of 
$\{s\},t_1,\dots,t_k$ with $\sum_{j=1}^k\tau_j=t_1+\dots+t_k$. 

In (\ref{residue}) we can use 
Lemma \ref{nabla-com} to express $P_{\gamma}(\hat{\nabla})$ as 
a finite sum of $P_{\gamma'}(x) P_{\gamma''}(\nabla)$. The polynomial 
$P_\gamma'(x)$ can (under the trace) be moved into 
$L_\star(f)$. We change the variables 
$t_1=t (1-u_2)(1-u_3)\cdots (1-u_k)$,
$t_2=t u_2(1-u_3)\cdots (1-u_k)$,
$t_3=t u_3(1-u_4)\cdots (1-u_k)$,\dots 
$t_{k-1}=t u_{k-1}(1-u_k)$ and $t_k=tu_k$ with 
Jacobian $t^{k-1}(1-u_3)(1-u_4)^2\cdots 
(1-u_k)^{k-2}$ and obtain
\begin{align}
\label{trX}
\mathrm{Tr}(X)&=\frac{\int_{[0,1]^N} ds \;P(s)}{\Gamma(\frac{|m|+z}{2}+|\beta|)}
\Big(\prod_{j=1}^{-1}
\frac{\Gamma(\frac{m_j}{2}+\beta_j)}{\Gamma(\frac{m_j}{2})}
\Big) \frac{\Gamma(\frac{m_k+z}{2}+\beta_k)}{\Gamma(\frac{m_k+z}{2})}
\int_0^\infty dt \,e^{-t}\,t^{\frac{|m|+z}{2}+|\beta|-1}
\\*
&\qquad \times \; 
\mathrm{tr}_{\mathbb{C}^{2^d}}\big(M_1\cdots M_k e^{(-1)^\bullet 
\tilde{\Omega}t\Sigma}\big)\;
\mathrm{Tr}\Big(L_\star\big(x^{\gamma_0} \star \partial^{\gamma_1} 
f_1 \star \cdots \star 
\partial^{\gamma_k} f_k\big) 
P_{\gamma_1,\dots,\gamma_k }(\nabla) \, e^{-t H}\Big)\;. \nonumber
\end{align}
Note that $\frac{\Gamma(\frac{m_k+z}{2}+\beta_j)}{
\Gamma(\frac{m_k+z}{2})}$ is a polynomial in $z$ of degree $\beta_k$. 

The traces $\mathrm{Tr}(L_\star(f) \nabla_{\mu_1}\cdots 
\nabla_{\mu_r} e^{-t H})$ are computed in
Proposition~\ref{Proposition-trace}. Accordingly, 
they have, up to a remainder which leads to a holomorphic function in
$z$, an asymptotic expansion 
$$
\mathrm{Tr}(L_\star(f) \nabla_{\mu_1}\cdots 
\nabla_{\mu_{|\gamma|}} e^{-t H})
= \sum_{a =0}^{N''} t^{-d/2-[\frac{|\gamma|}{2}]+a} \int_{\mathbb{R}^d} dx\;
f(x) P^a_{|\gamma|}(x)\;,
$$
where $P^a_{|\gamma|}(x)$ is a $\star$-polynomial of degree $\leq |\gamma|+2a$. 
Inserted into (\ref{trX}), the $t$-integral of
any such term yields (together with 
$\mathrm{tr}_{\mathbb{C}^{2^d}}\big(M_1\cdots M_k e^{(-1)^\bullet 
\tilde{\Omega}t\Sigma}\big)$)
a linear combination of 
\[
\frac{\Gamma(\frac{|m|+2|\beta|-2[|\gamma|/2]+2a+z-d}{2})}
{\Gamma(\frac{|m|+2\beta+z}{2})} \int_{\mathbb{R}^d} dx\;
f(x) P_{|\gamma|}^a(x)\;.
\]
As a function of $z\in\mathbb C$, the latter is trivially holomorphic
in $\mathbb C\setminus \mathbb{Z}$.  As $|m|+|\beta| \geq |\gamma|$,
this function is also holomorphic for $z>d$. If $z=d-N$ there is a
finite number of parameters $\beta_j,m_j,a, \gamma_j$ for which
$\frac{|m|+2|\beta|-2[|\gamma|/2]+2a-N}{2}$ is a non-positive integer
smaller than $\frac{|m|+2\beta+d-N}{2}$.  Precisely these parameters
yield simple poles at $z=d-N$. Each residue has the claimed structure.

We estimate the remainders (\ref{remainder}) in trace norm by:
\begin{align*}
&\frac{\|M_1\| \cdots\|M_k\|}{|\Gamma(\frac{m_1}{2})
\cdots\Gamma(\frac{m_{k}+z}{2})|}
\int dt_1\cdots dt_k \;t_1^{\frac{m_1}{2}+\beta_1'-1}\cdots 
t_{k-1}^{\frac{m_{k-1}}{2}+\beta'_{k-1}-1} t_k^{\frac{m_k+z}{2}+\beta'_k-1}\; 
\\
&
\qquad\qquad\qquad\qquad\times 
\int_{[0,1]^{N'}} ds\; |P'(s)|\Big\| \prod_{j=1}^k 
L_\star(\partial^{\gamma_j'} f_j)e^{(-1)^\bullet \tau_j
 \tilde{\Omega}\Sigma} 
P_{\gamma_j'}(\nabla) e^{-\tau_j (H+1)}\Big\|_1\;.
\end{align*}
By spectral theory and Corollary \ref{nabla-com2}, we get  
for $p\in[1,\infty)$:
$$
\|P_{\gamma_j'}(\nabla) e^{-\tau_j (H+1)}\|_p\leq C 
(\epsilon\tau_j)^{-\gamma_j'/2}\|e^{-\tau'_j (H+1)(1-\epsilon)}\|_p\;.
$$
Since the $\tau_j$ are linear in $t_j$ with $\sum_{j=1}^k\tau_j=\sum_{j=1}^kt_j$, 
setting $t=\sum_{j=1}^kt_j$, the H\"older
inequality gives
\begin{align*}
&\Big\| \prod_{j=1}^k 
L_\star(\partial^{\gamma_j'} f_j)e^{(-1)^\bullet \tau_j \tilde{\Omega}\Sigma}
P_{\gamma_j'}(\nabla) e^{-\tau'_j (H+1)}\Big\|_1\\
&\quad\leq
\prod_{j=1}^k\|L_\star(\partial^{\gamma_j'} f_j) \| \;
\|e^{(-1)^\bullet \tau_j \tilde{\Omega}\Sigma}\|_{t/{\tau}_j}
\|P_{\gamma_j'}(\nabla) e^{-\tau_j (H+1)}\|_{t/{\tau'}_j}
\\
&\quad\leq C e^{-t(1-\epsilon)}
\prod_{j=1}^k (\epsilon t_j)^{-\gamma_j'/2} 
\|L_\star(\partial^{\gamma_j'} f_j)\|  \;(2
\cosh(\tilde{\Omega}t))^{\frac{d\tau_j}{t}}
(2 \sinh (\tilde{\Omega}t(1-\epsilon))^{-\frac{d \tau_j}{t}}
\\
&\quad =C e^{-t(1-\epsilon-\tilde{\Omega}\epsilon d)}
\Big(\prod_{j=1}^k (\epsilon t_j)^{-\gamma_j'/2}
\|L_\star(\partial^{\gamma_j'} f_j)\|  \Big) \Big(\frac{
\cosh(\tilde{\Omega}t)}{
\sinh (\tilde{\Omega}t(1-\epsilon)) e^{\tilde{\Omega}t\epsilon} }\Big)^d\;,
\end{align*}
where results of Proposition \ref{Proposition-trace} and Lemma
\ref{bounds0} have been used.
Hence for $\epsilon':=(1-\epsilon-\tilde{\Omega}\epsilon d)>0$ 
the remainders (\ref{remainder}) are bounded in trace norm by
\begin{align*}
&\frac{C'}{|\Gamma(\frac{m_1}{2})
\cdots\Gamma(\frac{m_{k}+z}{2})|}\!
\int\!\! dt_1\cdots dt_k \;t_1^{\frac{m_1}{2}+\beta_1'-1-\frac{\gamma_1'}2}\cdots 
t_k^{\frac{m_k+z}{2}+\beta'_k-1-\frac{\gamma_k'}2}\; 
(t_1\!+\cdots +\!t_k)^{-d} e^{-\epsilon'(t_1+\dots+t_k)}\;.
\end{align*}
Remember that $|\gamma'|\leq|\beta'|+|m|$ and $|\beta'|>K$ and that
$K$ can be chosen as big as one wishes (by pushing the expansion over
and over).  So given $M\leq 2d$, by choosing $K >M/2+d$, we see that
the remainder terms are well defined as a trace-class operators for
$\Re(z) >M$. A similar analysis involving the $z$-derivative of the
remainders can be done, showing that by pushing the expansion far
enough, the remainders yield holomorphic contributions for $\Re(z)>
M$, with $M\in\mathbb R$ arbitrary.
\end{proof}

\section{The spectral action}
\label{SA}

\subsection{Generalities  }

For a unital spectral triple with real structure $(\A,\mathcal
H,\mathcal D,J)$, according to the spectral action principle
\cite{Connes:1996gi,Chamseddine:1996zu}, the bosonic action should
depend only on the spectrum of the fluctuated Dirac operator
$$
\mathcal D\mapsto \mathcal D_A:=\mathcal{ D}+A+\varepsilon' JA J^{-1}\;,
$$
where $A=\sum a_i[\mathcal D,b_i]$, $a_i,b_i\in\mathcal A$ (finite
sum), is a self-adjoint one-form and $\varepsilon'=\pm 1$ depending on
the KO-dimension of the triple. Ideally, such an action functional (of
$\mathcal D$ and of $A$) should be defined as the number of
eigenvalues of $\mathcal D_A^2$ smaller than a given scale
$\Lambda>0$:
$$
S_\Lambda(\mathcal{D}_A)=\sharp\big\{\lambda_n \,:\,
\lambda_n\in{\rm Spect}(\mathcal D_A^2),\, \lambda_n\leq \Lambda\big\}\;,
$$
or akin to the same and with $\chi$ the characteristic function of the
interval $[0,1]$:
\begin{align}
\label{Action}
S_\Lambda(\mathcal{D}_A)=\mathrm{Tr}\big(\chi(\mathcal{D}_A^2/\Lambda^2)\big)\;.
\end{align}
Then, the diverging part of $S_\Lambda(\mathcal{D}_A)$ in the limit
$\Lambda\to\infty$ should give access to an effective action
describing low energy physics. The problem is that with the
characteristic function, the expression \eqref{Action} may not have a
well-defined power series expansion in the limit $\Lambda\to\infty$. To
overcome this difficulty one uses, instead of the characteristic
function, a smooth one approximating it and being the inverse Laplace
transform of a Schwartz function on $\mathbb R^*_+$. By Laplace
transformation one then has
\begin{align*}
S_\Lambda(\mathcal{D}_A)=\int_0^\infty dt\; 
\mathrm{Tr}(e^{-t \mathcal{D}_A^2/\Lambda^2})\, \hat{\chi}(t)\;,  
\end{align*}
where $\hat{\chi}$ is the inverse Laplace transform of $\chi$. Assuming that the trace
of the heat kernel has an asymptotic expansion
\begin{align*}
\mathrm{Tr}(e^{-t \mathcal{D}_A^2}) 
= \sum_{k=-n}^\infty a_k(\mathcal{D}_A^2)\, t^{k}\;,\qquad 
n \in \mathbb{N}\;,
\end{align*}
we obtain
\begin{align}
\label{specaction-gen}
S_\Lambda(\mathcal{D}_A)= \sum_{k=-n}^\infty a_k(\mathcal{D}_A^2)\,\Lambda^{-2k}\,
\int_0^\infty dt\; t^{k} \hat{\chi}(t)\;.
\end{align}
One easily finds 
\begin{align*}
\int_0^\infty dt\; t^{k} \hat{\chi}(t) =\begin{cases} \frac{1}{\Gamma(-k)}\int_0^\infty
ds\;s^{-k-1} \chi(s)\,, & \text{for } k \notin \mathbb{N}\;,
\\
(-1)^{k}  
\chi^{(k)}(0)\,, & \text{for } k \in \mathbb{N} \;.
\end{cases}
\end{align*}
If one only wants to keep the non-vanishing terms in the
power-$\Lambda$ expansion as $\Lambda\to \infty$, it is therefore
sufficient to identify the the non-vanishing terms in the power-$t$
expansion of $\mathrm{Tr}(e^{-t \mathcal{D}_A^2})$ as $t\to 0$.

\medskip

If the spectral geometry $(\A,\mathcal H,\mathcal D,J)$ is 
non-unital, then the expression \eqref{Action} becomes ill-defined.
There are different ways to `regularize' the spectral action in this
case. For instance one may consider instead
\begin{align*}
S_\Lambda(\mathcal{D}_A):=\mathrm{Tr}\big(\chi(\mathcal{D}_A^2/\Lambda^2)
-\chi(\mathcal{D}^2/\Lambda^2)\big)\;.
\end{align*}
But the problem is then that one loses a lot of physical information
since one cannot access in this way the Einstein-Hilbert action.
Another possibility, used in \cite{Gayral:2004ww}, is to introduce a
supplementary (adynamic) scalar field $\rho\in\A$ and to define
\begin{align*}
S_\Lambda(\mathcal{D}_A,\rho)
:=\mathrm{Tr}\big(\rho\,\chi(\mathcal{D}_A^2/\Lambda^2)\big)\;.
\end{align*}
The advantage of this scheme is that one keeps the physical
interpretation by performing on the field equations the adiabatic
limit $\rho\to 1$ and that one can choose the one-form $A$ coming from
the unitization of $\A$ and not necessarily from $\A$ itself. However, in
the case of the Moyal spectral triple with ordinary Dirac operator, treated in
\cite{Gayral:2004ww}, the full computation with the real structure was
not possible; only the spectral action for partially fluctuated Dirac
operator $\mathcal D\mapsto \mathcal{ D}+A$ was evaluated. The last
possibility spelled out in \cite{Chamseddine:2005zk} is to replace
the scale $\Lambda$ by a dilaton field. That is, one performs the
replacement
$$
\Lambda\mapsto e^{-\phi},\quad \phi^*=\phi\in\A\;,
$$
and one leaves \eqref{Action} as it was:
\begin{align*}
S_\phi(\mathcal{D}_A)
:=\mathrm{Tr}\big(\chi(e^{\phi}\,\mathcal{D}_A^2\,e^{\phi})\big)\;.
\end{align*}
Although this expression is analytically well-defined and conceptually
perfect, the explicit computation of such a functional seems to be
fairly inaccessible, except for the commutative (manifold) case.

\medskip

In our setting of a spectral triple for Moyal plane with harmonic
propagation, the question of the definition and the computation of the
spectral action is way more easy. This is because even if the spectral
triple $(\A_\star,\mathcal H,\mathcal D_\bullet)$, $\bullet=1,2$, is
non-unital, the heat operator $e^{-t\D^2_\bullet}$ is trace-class for
all $t>0$ (see Lemma \ref{normalization}).  This means that the
definition \eqref{Action} of the spectral action for unital spectral
triple is still adapted to our situation.

In the next subsections we will perform a complete computation of the
spectral action for a $U(1)$-Higgs model for $d=4$. Before this, we
will derive a generic heat kernel type expansion when one tensorizes
$(\A_\star,\mathcal H,\mathcal D_\bullet)$ with a finite spectral triple.

\subsection{Heat kernel expansion in dimension four}

We derive here a short-time heat-kernel expansion for the semi-group
generated by the square of a twisted harmonic Dirac operator, for the
algebra of Schwartz functions with Moyal product. We start with
preliminary results on Schatten norm estimates, using
the estimate of Lemma \ref{bounds2} together with complex interpolation
methods. Here we specify to the case $d=4$.
\begin{Proposition}
\label{bounds3}
Let $f\in\A_\star$. Then for all $1\leq p\leq\infty$ and $t\in(0,1]$, we have$$
\|L_\star(f) e^{-t\mathcal D_\bullet ^2}\|_p\leq  \|f\|_2^{1-1/p}\, C(f)^{1/p}\,p^{-2/p}\,t^{-2/p}\;,
$$
where $C(f)$ is the constant appearing in Lemma \ref{bounds2}.
\end{Proposition}
\begin{proof}
 For $f\in\A_\star$, $t\in(0,1]$ and $1\leq p<\infty$, consider on
 the strip $S:=\{z\in\mathbb C:\Re(z)\in[0,1]\}$ the operator-valued
 function
$$
F_p:z\mapsto L_\star(f)e^{-tpz\D_\bullet^2}\;.
$$
The function $F_p$ is continuous on $S$, holomorphic on its interior
and by Lemma \ref{bounds2} it satisfies for $y\in\mathbb R$:
$$
\|F_p(iy)\|\leq \|f\|_2,\qquad \|F_p(1+iy)\|_1\leq C(f) (pt)^{-2}\;.
$$
Then, by standard complex interpolation methods (see for example
\cite{SimonTrace}) we have $F_p(z) \in \mathcal L^{1/\Re (z)}(\mathcal
H)$ for all $z \in S$ with
\begin{align*}
\|F_p(z)\|_{1/\Re (z)}
&\leq \|F_p(0)\|_\infty^{1-\Re (z)}\, \|F_p(1)\|_1^{\Re (z)}\leq  
\|f\|_2^{1-\Re(z)}  C(f)^{\Re(z)} (pt)^{-2\Re(z)}\;.
\end{align*}
Applying this for $z=1/p$, we get
$$
\|L_\star(f) e^{-t\mathcal D_\bullet^2}\|_p\leq  \|f\|_2^{1-1/p}\, 
C(f)^{1/p}\,p^{-2/p}\,t^{-2/p}\;,
$$
as needed.
\end{proof}

\begin{Remark}
\label{rem}
For $f\in\A_\star$, making a recursive chose of factorization as follows:
$$
f=f_1\star f_2,\quad \partial_{\mu_1}f_2=f_{1,\mu_1}\star f_{2,\mu_1},\quad \cdots  \quad\partial_{\mu_4} f_{2,\mu_1\mu_2\mu_3}=
f_{1,\mu_1\mu_2\mu_3\mu_4}\star f_{2,\mu_1\mu_2\mu_3\mu_4}\;,
$$
the constant $C(f)$ appearing in Lemma \ref{bounds2} (for $d=4$) and Proposition
\ref{bounds3} is a finite multiple (depending only on $\tilde\Omega$
and $\Theta$) of
\begin{align*}
&\|\bar f_1\star f_1\|_1^{1/2}\|\bar f_2\star f_2\|_1^{1/2}
+\|f_1\|_2\sum_{\mu_1=1}^4\Big(\|\bar f_{1,\mu_1}\star
f_{1,\mu_1}\|_1^{1/2}
\|\bar f_{2,\mu_1}\star f_{2,\mu_1}\|_1^{1/2}
+\|f_{1,\mu_1}\|_2\sum_{\mu_2=1}^4\Big(
\\
& \|\bar f_{1,\mu_1\mu_2}\star f_{1,\mu_1\mu_2}\|_1^{1/2}
\|\bar f_{2,\mu_1\mu_2}\star f_{2,\mu_1\mu_2}\|_1^{1/2}
+\|f_{1,\mu_1\mu_2}\|_2\sum_{\mu_3=1}^4\Big(
\|\bar f_{1,\mu_1\mu_2\mu_3}\star f_{1,\mu_1\mu_2\mu_3}\|_1^{1/2}
\\
&\times \|\bar f_{2,\mu_1\mu_2\mu_3}\star f_{2,\mu_1\mu_2\mu_3}\|_1^{1/2}+
\|f_{1,\mu_1\mu_2\mu_3}\|_2\sum_{\mu_4=1}^4\Big(
\|\bar f_{1,\mu_1\mu_2\mu_3\mu_4}\star
f_{1,\mu_1\mu_2\mu_3\mu_4}\|_1^{1/2}
\\
&\times\|\bar f_{2,\mu_1\mu_2\mu_3\mu_4}\star f_{2,\mu_1\mu_2\mu_3\mu_4}\|_1^{1/2}
+\|f_{1,\mu_1\mu_2\mu_3\mu_4}\|_2\|f_{2,\mu_1\mu_2\mu_3\mu_4}\|_2\Big)
\Big)\Big)\Big)\;.
\end{align*}
\end{Remark}

\begin{Lemma}
\label{bounds4}
Let $f\in\A_\star$, $1\leq p\leq\infty$, $t\in(0,1]$ and $k\in\mathbb N$. Then, there exists a finite constant $C_{p,k}(f)$ such that
$$
\|L_\star(f) \D_\bullet^ke^{-t\mathcal D_\bullet^2}\|_p\leq C_{p,k}(f) \,t^{-2/p-k/2}\;.
$$
\end{Lemma}
\begin{proof}
By spectral theory, we have  $\|\D_\bullet^ke^{-t\mathcal{D}^2_\bullet}\|=(k/2et)^{k/2}$, so the proof is a consequence of Proposition \ref{bounds3}.
\end{proof}
The next Lemma will explain why there is a major difference in the
spectral action when perturbing $\D$ by $A+JAJ^{-1}$ or simply by
$A$.
\begin{Lemma}
\label{bounds5}
For  $f,g\in\A_\star$, the operator $L_\star(f)\,R_\star(g)$ is of
trace class on $\mathcal H$.
\end{Lemma} 
\begin{proof}
By factorization, we can find $f_1,f_2,g_1,g_2\in\A_\star$ such that
$$
f=f_1\star f_2\,,\quad g=g_1\star g_2\;.
$$
Hence 
$$
L_\star(f)\,R_\star(g)=L_\star(f_1\star f_2)\,R_\star(g_1\star g_2)=L_\star(f_1)L_\star( f_2)\,R_\star(g_1)R_\star( g_2)\;.
$$
But since the left and right regular representations commute (by
associativity of the Moyal product), we get
$$
L_\star(f)\,R_\star(g)=L_\star(f_1)R_\star(g_1)\,L_\star( f_2)R_\star( g_2)\;,
$$
so that it suffices to show that $L_\star(f)\,R_\star(g)$ is
Hilbert-Schmidt for all $f,g\in\A_\star$.  From the operator kernel
formula \eqref{Lstarf} of $L_\star(f)$, and a similar one for
$R_\star(g)$, one easily deduces the operator kernel for the product
$L_\star(f)\,R_\star(g)$, and after a few lines of computations, we
get for a suitable constant depending only on $\det(\Theta)$:
$$
\|L_\star(f)\,R_\star(g)\|_2^2
=\int dx\,dy\, \big|[L_\star(f)\,R_\star(g)](x,y)\big|^2
= C \|f\|_2^2\,\|g\|_2^2\;.
$$
This completes the proof.
\end{proof}
\begin{Corollary}
\label{bounds6}
Let $\nabla_\mu^a$, $\mu=1,\cdots,4$, be the operators on 
$L^2(\mathbb R^4)$ given by
\begin{align*}
\nabla_\mu^{a}  :=i\partial_\mu +a_{\mu\nu} x^\mu\;,
\quad 
a\in M_4(\mathbb R)\;.
\end{align*}
Then for $f,g\in\A_\star$, $t\in(0,1]$  and 
$P_\alpha(\hat{\nabla})$ a polynomial of order
$\alpha$ in the operators $\nabla_\mu^{a} $, there exists a finite
constant $C(f,g,\alpha)$ such that
$$
\big\|L_\star(f)R_\star(g)P_\alpha(\hat{\nabla})e^{-tH}\big\|_1
\leq C(f,g,\alpha)\, t^{-\alpha/2}\;.
$$ 
\end{Corollary}
\begin{proof}
 From Lemma \ref{bounds5}, it suffices to show that
 $\|P_\alpha(\hat{\nabla})e^{-tH}\|\leq Ct^{-\alpha/2}$, which will follow
 by spectral theory if $P_\alpha(\hat{\nabla})(1+H)^{-|\alpha|/2}$ is
 bounded. But this is a slight generalization of Corollary
 \ref{nabla-com2}.
\end{proof}
\noindent
We can now deduce the germ of the asymptotic expansion formula we need.
\begin{Proposition}
\label{germ}
Let $(\mathcal A_{\bf f},\mathcal H_{\bf f},\mathcal D_{\bf f}, J_{\bf
f})$ be a finite spectral triple. Let $\D:=\D_\bullet\otimes
1+\Gamma\otimes\mathcal D_{\bf f}$, $\bullet=1,2$, be the Dirac operator of the
product spectral triple $(\A\otimes\mathcal A_{\bf f},\mathcal
H\otimes\mathcal H_{\bf f},\D_\bullet\otimes 1+\Gamma\otimes\mathcal D_{\bf
f})$. Let also $\D_A:=\D+A+{\bf J} A {\bf J}^{-1}$ be the fluctuated
Dirac operator. Here ${\bf J}:=J\otimes J_{\bf f}$ and $A$ is a
self-adjoint one-form, that is $A=A^*:=\sum_i a_i[\D,b_i]$, where the
sum is finite and $a_i,b_i\in \A\otimes\A_{\bf f}$. 
In terms of the decomposition
$\mathcal{D}_A^2=\mathcal{D}^2 + F_0+F_1+ {\bf J}(F_0+F_1) {\bf J}^{-1} + 
2A{\bf J}A{\bf J}^{-1}$, where $F_0$ is a bounded operator and $F_1$ 
is linear in the operators $\nabla^a_\mu$ of Corollary~\ref{bounds6},
the following holds:
\begin{align*}
\mathrm{Tr}\big(e^{-t\mathcal{D}^2_A}\big)&=
\mathrm{Tr}\Big(\Big\{
1
-2 t(F_0+F_1)+t^2\big(F_0^2+F_1F_0+F_0F_1+F_1^2\big)
\\
&\quad-\frac{t^3}{3}\big(
F_0[\mathcal D^2,F_1] -[\mathcal D^2,F_1] F_0 
+F_1[\mathcal D^2,F_1]
+F_0 F_1^2 + F_1F_0F_1+F_1^2 F_0+F_1^3
\big)
\nonumber
\\
& \quad+ \frac{t^4}{12}
\big(
F_1[\mathcal{D}^2,[\mathcal D^2,F_1]]
+ 2 F_1^2\,[\mathcal{D}^2,F_1]
+F_1\,[\mathcal{D}^2,F_1]F_1
+F_1^4\big)
\Big\}
e^{-t\mathcal D^2}\Big)+ \mathcal{O}(\sqrt{t})\;.
\end{align*}
\end{Proposition}
\begin{proof}
First, it is clear that $e^{-t\D_A^2}$ is of trace-class for all
$t>0$. Indeed, since the eigenvalues of $\D_\bullet$ behave like
$n^{-1/8}$, its resolvent belongs to the Schatten ideal $\mathcal
L^{8+\varepsilon}(\mathcal H)$ for all $\varepsilon>0$. From the
relation
$$
\frac1{\D+i}=\frac1{\D_\bullet\otimes 1+i}\Big(1-(A+{\bf J}A{\bf J}^{-1}+\Gamma\otimes\mathcal D_{\bf f})\frac1{\D+i}\Big)\;,
$$
and the fact that $A+{\bf J}A{\bf J}^{-1}+\Gamma\otimes\mathcal D_{\bf
f}$ is bounded, we see that the resolvent of $\D$ belongs to
$\mathcal L^{8+\varepsilon}(\mathcal H\otimes\mathcal H_{\bf f})$ for
all $\varepsilon>0$ too.  Accordingly, $e^{-t\D_A^2}$ is of
trace-class for all $t>0$.

Note also that the bounds of Lemmas \ref{bounds1}, \ref{bounds2},
\ref{bounds4} and Proposition \ref{bounds3}, \ref{bounds6} remain valid
with $\D$ instead of $\D_\bullet$. Indeed, since
$\{\D_\bullet,\Gamma\}=0$ and $\Gamma^2=1$, we get
$\D^2=\D_\bullet^2\otimes 1+1\otimes\D_{\bf f}^2$ and thus
$$
\D^ke^{-t\D^2}=\sum_{j=0}^k C_{k,j}\Gamma^{k-j}\D_\bullet^j 
e^{-t\D_\bullet^2}\otimes \D_{\bf f}^{k-j}e^{-t\D_{\bf f}^2}\;.
$$
This implies for $f,g\in\B_\star$, $a,b\in\A_{\bf f}$ and $1\leq p\leq \infty$,
\begin{align*}
&\|L_\star(f)\otimes a\,{\bf J}L_\star(g)\otimes b{\bf J}^{-1}\, 
\D^ke^{-t\D^2}\|_p
\\*
&\qquad\qquad\qquad\qquad\leq\sum_{j=0}^k |C_{k,j}|
\|L_\star(f)\,R_\star(\overline g)\,\D_\bullet^je^{-t\D_\bullet^2}\|_p\|aJ_{\bf f}bJ_{\bf f}^{-1} \D_{\bf f}^{k-j}e^{-t\D_{\bf f}^2}\|_p\;.
\end{align*}
Thus, we may assume without loss of generality that there is no
finite spectral triple in the picture.

We are going to deduce the expansion from the Duhamel principle:
$$
e^{-t(A+B)}=e^{-tA}-t\int_0^1e^{-st(A+B)}\,B\,e^{-(1-s)tA}\,ds\;.
$$
We write $\D_A^2=\D^2+\tilde{F}_0+\tilde{F}_1$,
with $\tilde{F}_0:= F_0+{\bf J}F_0{\bf J}^{-1}+ 2A{\bf J}A{\bf
J}$ and $\tilde{F}_1:=
F_1+{\bf J}F_1{\bf J}^{-1}$. The operator $\tilde{F}_0$ is bounded,
whereas $\tilde{F}_1$ is unbounded but relatively $\D$-bounded.  The
Duhamel expansion allows us to write (formally first):
\begin{equation}
\label{Duhamel}
e^{-t \mathcal D^2_A}=\sum_{j=0}^\infty (-t)^j\,E_j(t)\;,
\end{equation}
where $E_0(t):=e^{-t\mathcal D^2}$ and for $j>0$:
\begin{align*}
E_j(t):=\sum_{i_1,\cdots,i_j\in\{0,1\}}\int_{\triangle_j} \,
e^{-s_0 t\mathcal D^2}\,\tilde{F}_{i_1}\,e^{-s_1t\mathcal D^2}\cdots
\tilde{F}_{i_j}\,e^{-s_jt\mathcal D^2}\,d^js\;,
\end{align*}
and  $\triangle_j$ denotes the ordinary $j$-simplex:
\begin{align*}
\triangle_j&:=\big\{s\in\mathbb R^{j+1};\, s_k\geq0,\,\sum_{k=0}^j s_k=1\big\}\;.
\end{align*}
We first show that the sum \eqref{Duhamel} converges in the trace norm
for small values of $t>0$. We only treat the case $j\geq1$, the case
$j=0$ being covered by Lemma \ref{bounds1}. For that we use the
H\"older inequality (since $\sum_{k=0}^j s_k=1$):
$$
\|E_j(t)\|_1\leq\!\! \!\!\!\sum_{i_1,\cdots,i_j\in\{0,1\}}
\int_{\triangle_j} \,\|e^{-s_0 t\mathcal D^2}\,\tilde{F}_{i_1}\,
e^{-s_1t\mathcal D^2}\|_{(s_0+s_1)^{-1}}\| \tilde{F}_{i_2}\,
e^{-s_2t\mathcal D^2}\|_{s_2^{-1}}\cdots
\|\tilde{F}_{i_j}\,e^{-s_jt\mathcal D^2}\|_{s_j^{-1}}\,d^js\;.
$$
Then we use the estimate of Proposition \ref{bounds3} and Lemma
\ref{bounds5} for $k=2,\cdots,j$ (see the Remark \ref{rem} for the
precise value of the constants):
$$
\|\tilde{F}_{i_k}\,e^{-s_kt\mathcal D^2}\|_{s_k^{-1}}\leq
\begin{cases}
4\|\tilde{F}\|^2C_1(\tilde{F})^{s_k} \,t^{-2s_k}\,,
\quad& \mbox{if}\quad i_k=0\;,\\
2\|\tilde{F}\|C_2(\tilde{F})^{s_k} \,t^{-2s_k} \,(ts_k)^{ -1/2}\,,
\quad& \mbox{if} \quad i_k=1\;.
\end{cases}
$$
For the case $i_k=1$, we need to use the factorization property of the
algebra of Schwartz functions with Moyal product (as in the proof of
Lemma \ref{bounds2}), to expand $A$ as a finite sum of products of
elements in $\A_\star\otimes M_{16}(\mathbb C)$, and then we can
proceed as for the other factors.  Taking into account that there are
$2^j$ such terms and that
$$
\int_{\Delta_j}\prod_{i=0}^j s_i^{-1/2}\,d^js\leq2^{j}\;,
$$
we get, since $\sum_{k=0}^j s_k=1$, the rough estimate 
$$
\|E_j(t)\|_1\leq 2^{j} 
\big(4\|\tilde{F}\|^{2}+4\|\tilde{F}\|\big)^j t^{-j/2-2}\;.
$$
Thus the sum $\sum_{j=0}^\infty (-t)^j\,E_j(t)$ converges absolutely
in the trace-norm for small values of $t$.  These estimates also show
that
\begin{align*}
\Big|{\rm Tr}\big(e^{-t \mathcal D^2_A}\big)
-\sum_{j=5}^\infty (-t)^j\,{\rm Tr}\big(E_j(t)\big)\Big|=O(t^{1/2})\,,
\quad t\to 0\;,
\end{align*}
and accordingly, we only need to consider the terms $(-t)^j{\rm
Tr}\big(E_j(t)\big)$ for $j=0,1,2,3$ and $4$.

Note first that
\begin{align*}
\mathrm{Tr}(-t E_1(t))= \int_0^1 \,\mathrm{Tr}\Big(-t\,e^{-s t\mathcal
D^2}\,(\tilde{F}_0+\tilde{F}_1)\,
e^{-(1-s)t\mathcal D^2}\Big)\,ds=\mathrm{Tr}\Big(-t\,(\tilde{F}_0
+\tilde{F}_1)\,e^{-t\mathcal D^2}\Big)\;.
\end{align*}
For $j=2,3$ and $4$, we use the relation \eqref{commm}
to collect the heat operators as follows:
\begin{align}
E_2(t) &= \int_{\triangle_2} \,e^{-s_1 t\mathcal D^2}\,(\tilde{F}_0+\tilde{F}_1)\,
e^{-(s_2-s_1)t\mathcal D^2}
(\tilde{F}_0+\tilde{F}_1)\,e^{-(1-s_2)t\mathcal D^2}\,ds_1 ds_2
\nonumber
\\
&= \int_{\triangle_2} \,e^{-s_1 t\mathcal D^2}\,(\tilde{F}_0+\tilde{F}_1)^2
e^{-(1-s_1)t\mathcal D^2}\,ds_1 ds_2
- t \int_{\triangle_2} \int_0^1 dr \,(s_2-s_1) 
\,e^{-s_1 t\mathcal D^2}\,(\tilde{F}_0+\tilde{F}_1)\,
\nonumber
\\
&\qquad\qquad \qquad \qquad \qquad \times e^{-(s_2-s_1)r t\mathcal D^2}
[\mathcal D^2,(\tilde{F}_0+\tilde{F}_1)]
e^{-(s_2-s_1)(1-r) t\mathcal D^2}
\,e^{-(1-s_2)t\mathcal D^2}\,ds_1 ds_2
\nonumber
\\
&= \int_{\triangle_2} ds_1 ds_2\,e^{-s_1 t\mathcal D^2}\,
\Big\{(\tilde{F}_0+\tilde{F}_1)^2 -t(s_2-s_1) (\tilde{F}_0+\tilde{F}_1)\,
[\mathcal D^2,(\tilde{F}_0+\tilde{F}_1)]\Big\}
e^{-(1-s_1)t\mathcal D^2}
\nonumber
\\
& \quad+ t^2 \int_{\triangle_2} ds_1 ds_2 \int_0^1 dr_1 dr_2 \,r_1 (s_2-s_1)^2 
\,e^{-s_1 t\mathcal D^2}\,(\tilde{F}_0+\tilde{F}_1)\,
\nonumber
\\
&\qquad \times 
e^{-(s_2-s_1)r_1r_2 t\mathcal D^2}
[\mathcal{D}^2,[\mathcal D^2,(\tilde{F}_0+\tilde{F}_1)]]
e^{-(s_2-s_1)(1-r_1r_2) t\mathcal D^2}
e^{-(1-s_2)t\mathcal D^2}
\nonumber
\\
&= \int_{\triangle_2} ds_1 ds_2\,e^{-s_1 t\mathcal D^2}\,
\Big\{(\tilde{F}_0+\tilde{F}_1)^2 -t(s_2-s_1) (\tilde{F}_0+\tilde{F}_1)\,
[\mathcal D^2,(\tilde{F}_0+\tilde{F}_1)]
\nonumber
\\
& \qquad\qquad + \frac{t^2}{2}(s_2-s_1)^2 (\tilde{F}_0
+\tilde{F}_1)[\mathcal{D}^2,[\mathcal D^2,(\tilde{F}_0+\tilde{F}_1)]]\Big\}
e^{-(1-s_1)t\mathcal D^2}
\nonumber
\\
& - t^3 \int_{\triangle_2} ds_1 ds_2 \int_0^1 dr_1 dr_2 dr_3\,
r_1^2 r_2 (s_2-s_1)^3 
\,e^{-s_1 t\mathcal D^2}\,(\tilde{F}_0+\tilde{F}_1)\,
e^{-(s_2-s_1)r_1r_2r_3 t\mathcal D^2}
\nonumber
\\*
&\qquad \qquad \times 
[\mathcal{D}^2,[\mathcal{D}^2,[\mathcal D^2,(\tilde{F}_0+\tilde{F}_1)]]]]
e^{-(s_2-s_1)(1-r_1r_2r_3)t \mathcal D^2}\,
e^{-(1-s_2)t \mathcal D^2}\;.
\end{align}
Since the principal symbol of $\D^2$ is scalar, we see that
$[\mathcal{D}^2,[\mathcal{D}^2,[\mathcal D^2,\tilde{F}_1]]]]$ has order $4$.
Thus, Lemma \ref{bounds4} shows that the last integral (multiplied by
its $t^2$ global prefactor) gives rise to a trace-class operator which
trace is of order $t^{1/2}$.  Integrating the trace over $\Delta_2$
and disregarding the terms that vanish when $t\to0$, we find
\begin{align*}
\mathrm{Tr}(t^2 E_2(t))=
\mathrm{Tr}\Big( \Big\{
\frac{t^2}{2} (\tilde{F}_0+\tilde{F}_1)^2 
&-\frac{t^3}{6} \big(\tilde{F}_0[\mathcal D^2,\tilde{F}_1]
+\tilde{F}_1[\mathcal D^2,\tilde{F}_0]
+\tilde{F}_1[\mathcal D^2,\tilde{F}_1]\big)
\nonumber
\\*
&+ \frac{t^4}{24}
\tilde{F}_1[\mathcal{D}^2,[\mathcal D^2,\tilde{F}_1]]\Big\}
e^{-t\mathcal D^2}\Big)+ \mathcal{O}(\sqrt{t})\;.
\end{align*}
It will be more convenient to write 
$\mathrm{Tr}\Big( F_1[\mathcal D^2,F_0]e^{-t\mathcal D^2}\Big)
=\mathrm{Tr}\Big( -[\mathcal D^2,F_1]F_0e^{-t\mathcal D^2}\Big)
$. 
By similar arguments one finds
\begin{align*}
\mathrm{Tr}(-t^3 E_3(t))
= \mathrm{Tr}\Big(
\Big\{&-\frac{t^3}{6} (\tilde{F}_0 \tilde{F}_1^2 
+ \tilde{F}_1\tilde{F}_0\tilde{F}_1+\tilde{F}_1^2
\tilde{F}_0+\tilde{F}_1^3)
\nonumber
\\
&+\frac{t^4}{24}(2 \tilde{F}_1^2\,[\mathcal{D}^2,\tilde{F}_1]
+\tilde{F}_1\,
[\mathcal{D}^2,\tilde{F}_1]\tilde{F}_1)\Big\}\,e^{-t\mathcal D^2}  \Big)
+\mathcal{O}(\sqrt{t})\;,
\end{align*}
and lastly
\begin{align*}
\mathrm{Tr}(t^4 E_4(t))
&= \mathrm{Tr}\Big(
\frac{t^4}{24} \tilde{F}_1^4\,e^{-t\mathcal D^2}  \Big)
+\mathcal{O}(\sqrt{t})\;.
\end{align*}
In summary, we have
\begin{align}
\label{Duhamel-4}
\mathrm{Tr}\big(e^{-t\mathcal{D}^2_A}\big)
&= \mathrm{Tr}\Big(\Big\{
1 -t(\tilde{F}_0+\tilde{F}_1)+\frac{t^2}{2}\big(\tilde{F}_0^2
+\tilde{F}_1\tilde{F}_0+\tilde{F}_0\tilde{F}_1+\tilde{F}_1^2\big)
\\
&\quad-\frac{t^3}{6}\big(
\tilde{F}_0[\mathcal D^2,\tilde{F}_1]
- [\mathcal D^2,\tilde{F}_1]\tilde{F}_0 
+\tilde{F}_1[\mathcal D^2,\tilde{F}_1]
+\tilde{F}_0 \tilde{F}_1^2 + \tilde{F}_1\tilde{F}_0\tilde{F}_1
+\tilde{F}_1^2 \tilde{F}_0+\tilde{F}_1^3
\big)
\nonumber
\\
& \quad+ \frac{t^4}{24}
\big(
\tilde{F}_1[\mathcal{D}^2,[\mathcal D^2,\tilde{F}_1]]
+ 2 \tilde{F}_1^2\,[\mathcal{D}^2,\tilde{F}_1]
+\tilde{F}_1\,[\mathcal{D}^2,\tilde{F}_1]\tilde{F}_1
+\tilde{F}_1^4\big)
\Big\}
e^{-t\mathcal D^2}\Big)+ \mathcal{O}(\sqrt{t})\;.\nonumber
\end{align}
Now, we can take into account the result of Lemma \ref{bounds5}, which
says in this context that mixed products $F_i{\bf J}F_j{\bf J}^{-1}$
are already trace-class. Since $JL_\star(g)J^{-1}=R_\star(\bar g)$,
with $R_\star$ the right regular representation, we see by Lemma
\ref{nabla-com} that all terms in \eqref{Duhamel-4} with products
of $F_i$ and ${\bf J}F_j{\bf J}^{-1}$ are (up to matrices) of the form
$$
L_\star(f)R_\star(g)\prod_{\mu=1}^4{(\hat{\nabla}_\mu)}^{\alpha_\mu}e^{-tH},
$$
with $|\alpha|$ not exceeding the number of $F_1$ plus the number of
commutators by $\D^2$. This argument also relies on the fact that
${\bf J}$ commutes with $\hat{\nabla}_\mu$ according to (\ref{Jab}).
Thus, Corollary \ref{bounds6} shows that the cross-terms, i.e. the
terms with powers of both $F_i$ and ${\bf J}F_j{\bf J}^{-1}$ resulting
from products of $\tilde{F}_0=F_0+ {\bf J}F_0{\bf J}^{-1}
+ 2 F_{-1} {\bf J}F_{-1}{\bf J}^{-1}$ and 
$\tilde{F}_1=F_1+ {\bf J}F_1{\bf J}^{-1}$, 
where $F_{-1}:=A$, give
rise to vanishing contributions in the limit $t\to 0$. Thus, only the
terms with either powers of $F_i$ or powers of ${\bf J}F_i{\bf
J}^{-1}$ do contribute to the diverging part of this asymptotic the
expansion.  Since moreover ${\bf J}$ commutes with $\D$ (we are in
even KO-dimension), the trace property shows that both terms (with
only $A$ or only ${\bf J}A{\bf J}^{-1}$) give the same contribution
and we get the announced result.
\end{proof}
\begin{Remark}
\label{conclusion}{\rm
A very important feature of Proposition \ref{germ} is that if 
the heat-trace of the partially fluctuated  Dirac operator
$\widetilde\D_A:=\D+A$, $A=\sum_i a_i[\D,b_i]$, has an asymptotic
expansion
$$
{\rm Tr}\big(e^{-t \widetilde\D_A^2}\big)=a_0 \,t^{-4}+\sum_{k=1}^4 a_k \,t^{-2+k/2}+
\mathcal{O}(\sqrt{t})\;,
$$ 
then the heat-trace of the fully fluctuated  Dirac 
operator $\D_A:=\D+A+{\bf J} A {\bf J}^{-1}$
has the asymptotic expansion
$$
{\rm Tr}\big(e^{-t \D_A^2}\big)=a_0 \,t^{-4}+2\sum_{k=1}^4 a_k \,t^{-2+k/2}+
\mathcal{O}(\sqrt{t})\;,
$$ 
for the same coefficients $a_0,\cdots,a_k$. Also, this shows that the
asymptotic expansion of the heat-trace of the fully fluctuated Dirac
operator is independent of the choice of the real structure $J_{\bf
 f}$ of the finite spectral triple $(\mathcal A_{\bf f},\mathcal
H_{\bf f},\mathcal D_{\bf f})$. This fact holds for Moyal spectral
triples with harmonic propagation in any (even) dimension. }
\end{Remark}

\subsection{Application: the spectral action for the $U(1)$-Higgs model}
\label{U1}

In the Connes-Lott spirit \cite{Connes:1990qp} we take the tensor
product of the 4-dimensional spectral triple
$(\mathcal{A}_\star,\mathcal{H},\mathcal{D}_\bullet,\Gamma,J)$,
$\bullet=1,2$, with the finite Higgs spectral triple
$(\mathbb{C}\oplus\mathbb{C}, \mathbb{C}^2, M\sigma_1,J_{\bf f})$,
where $M>0$ and $J_{\bf f}$ is any real structure.  The Dirac operator
$\mathcal{D}=\mathcal{D}_\bullet \otimes 1 + \Gamma \otimes M\sigma_1$
of the product triple becomes
\begin{align*}
\mathcal{D} =\left(\begin{array}{cc} 
\mathcal{D}_\bullet & M \Gamma \\ 
M \Gamma  & \mathcal{D}_\bullet
\end{array}\right)\;.
\end{align*}
In this representation, the algebra is $\mathcal{A}_\star \oplus
\mathcal{A}_\star \ni (f,g)$, which acts on $\mathcal{H}
\oplus \mathcal{H}$ by diagonal left Moyal multiplication. The
commutator of $\mathcal{D}$ with $(f,g)$ is, in case that
$\mathcal{D}_1$ is chosen, according to (\ref{Gamma}) given by 
\begin{align*}
[\mathcal{D},(f,g)] =\left(\begin{array}{cc} 
i \Gamma^\mu L_\star(\partial_\mu f)  & M \Gamma L_\star(g-f)\\ 
M \Gamma  L_\star (f-g) & i\Gamma^\mu 
L_\star(\partial_\mu g)
\end{array}\right)\;.
\end{align*}
(If we choose  $\mathcal{D}_2$ instead, then $\Gamma^\mu $ has to be
replaced by $\Gamma^{\mu +4}$ everywhere.)  This shows that the selfadjoint
fluctuation $A=\sum_i a_i [\mathcal{D},b_i]$ of the fluctuated Dirac
operators $\mathcal{D}_A= \mathcal{D} + A+{\bf J}A{\bf J}^{-1}$, ${\bf
J}=J\otimes J_{\bf f}$, is of the form
\begin{align*}
A  =\left(\begin{array}{cc} 
\Gamma^\mu L_\star(A_\mu)
& \Gamma L_\star(\phi) \\ \Gamma L_\star(\bar{\phi})  & 
\Gamma^\mu  L_\star(B_\mu)
\end{array}\right)\;,
\end{align*}
for real two real one-forms $A_\mu,B_\mu \in \mathcal{A}_\star$ and a
one complex field $\phi \in \mathcal{A}_\star$. 
Again, this holds for $\mathcal{D}_1$; for $\mathcal{D}_2$ we have to
replace  $\Gamma^\mu $ by $\Gamma^{\mu +4}$.

In terms of the connection introduced in Lemma~\ref{nabla-com} and
using (\ref{DGamma}) we identify the relevant operators arising in the
expansion $\mathcal{D}_A^2=\mathcal{D}_\bullet^2+ M^2 +(F_0+F_1)+{\bf
J}(F_0+F_1){\bf J}^{-1}+ 2A{\bf J}A{\bf J}^{-1}$ of Proposition
\ref{germ} as follows:
\begin{align*}
F_0 &=   \left(\begin{array}{cc} 
L_\star(V_{A,\phi})1
+ \frac{i}{4}[\Gamma^\mu,\Gamma^\nu]  L_\star(F^A_{\mu\nu}) 
&
i \Gamma^\mu \Gamma L_\star(D_\mu \phi) 
\\[1ex]
i \Gamma^\mu \Gamma
L_\star(\overline{D_\mu \phi}) & L_\star(V_{B,\phi})1
+ \frac{i}{4}[\Gamma^\mu,\Gamma^\nu] L_\star(F^B_{\mu\nu}) 
\end{array}\right)  \;,
\\
F_1 &=   \left(\begin{array}{cc} 
2i L_\star(A^\mu) \nabla_\mu & 0 \\ 
0 & 2i L_\star(B^\mu) \nabla_\mu
\end{array}\right)\;,
\end{align*}
where
\begin{align*}
V_{A,\phi}&:=\phi \star \bar{\phi} +M(\phi+ \bar{\phi})
+ (g^{-1})^{\mu\nu} (i\partial_\mu A_\nu+A_\mu\star A_\nu)\;,\\
V_{B,\phi} &:= \bar{\phi} \star \phi  +M(\phi+ \bar{\phi})
+ (g^{-1})^{\mu\nu} (i\partial_\mu B_\nu+B_\mu\star B_\nu)\;,\\
F^A_{\mu\nu}&: = \partial_\mu A_\nu - \partial_\nu A_\mu -i(A_\mu \star
A_\nu - A_\nu \star A_\mu)\;,\\
F^B_{\mu\nu}&: = \partial_\mu B_\nu - \partial_\nu B_\mu -i(B_\mu \star
B_\nu - B_\nu \star B_\mu)\;,\\
D_\mu \phi &:= \partial_\mu \phi -i A_\mu\star \phi+i \phi \star
B_\mu -i M(A_\mu-B_\mu)\;.
\end{align*}
We have used
\begin{align*}
\mathcal{D}_1\Gamma^\sigma L_\star(A_\sigma)+
L_\star(A_\sigma)\mathcal{D}_1\Gamma^\sigma 
&= -\Gamma^\sigma [\mathcal{D}_1, L_\star(A_\sigma)]
+ 2i (g^{-1})^{\mu\nu} L_\star (\partial_\mu A_\nu) 
+ 2i L_\star(A^\mu) \nabla_\mu \;.
\end{align*}
According to the general asymptotic expansion we have obtained in
 Proposition \eqref{germ}, the only further commutators we need are
 $[\D^2,F_1]$ and $[\D^2,[\D^2,F_1]]$. Their expression
 will easily follows from the following computation which relies on
 the relations in Lemma~\ref{nabla-com}:
\begin{align*}
&[\mathcal{D}^2,2i L_\star(A^\mu) \nabla_\mu ]
= 
-2i  L_\star\big((g^{-1})^{\rho\sigma} \partial_\rho\partial_\sigma
A^\mu\big)  \nabla_\mu 
-4i L_\star(\partial^\nu A^\mu) \nabla_\nu 
\nabla_\mu 
+ 4 \tilde{\Omega}^2 \Theta_{\mu\nu} L_\star(A^\mu) \tilde{\nabla}^{\nu}\;,
\\
&[\mathcal{D}^2,[\mathcal{D}^2,2i L_\star(A^\mu) 
\nabla_\mu ]
= 
8 i 
L_\star(\partial^\rho \partial^\nu A^\mu) \nabla_\rho 
\nabla_\nu \nabla_\mu + \text{ lower order}\nonumber\;,
\end{align*}

We have already computed the matrix trace of $e^{-t\tilde{\Omega} \Sigma}$
in Lemma \ref{bounds1}, and the result is $16\cosh^4 (\tilde{\Omega}t)$.
Using
\begin{align*}
-i\Sigma&= (ib_\mu-ib_\mu^*)(b^\mu+b^{\mu*})
=g_{\mu\nu}\Big(\Gamma^\nu
+\frac{\tilde{\Omega}}{2}\Theta^{\nu\rho}\Gamma_{\rho+4}\Big) 
g^{\mu\sigma}\Big(\Gamma_{\sigma+4}
+\frac{\tilde{\Omega}}{2}\Theta_{\sigma\tau}\Gamma^{\tau}\Big) \;,
\end{align*}
the other matrix traces follow from the Clifford algebra:
\begin{align*}
\mathrm{tr}_{\mathbb{C}^{16}}\Big(\frac{i}{4} [\Gamma^\mu,\Gamma^\nu]
\cdot e^{-t\tilde{\Omega} \Sigma}\Big) &= 
-8 \tilde{\Omega} \Theta^{\mu\nu} t + \mathcal{O}(t^2)\;,
\nonumber
\\
\mathrm{tr}_{\mathbb{C}^{16}}\Big(\frac{i}{4} [\Gamma^\mu,\Gamma^\nu]
\cdot \frac{i}{4} [\Gamma^\rho,\Gamma^\sigma]\cdot 
e^{-t\tilde{\Omega} \Sigma}\Big) &= 
4(g^{-1})^{\mu\rho}(g^{-1})^{\nu\sigma}-
4 (g^{-1})^{\mu\sigma}(g^{-1})^{\nu\rho}\big) + \mathcal{O}(t)
\nonumber
\\
\mathrm{tr}_{\mathbb{C}^{16}}\big(
i \Gamma^\mu \Gamma 
\cdot i \Gamma^\nu \Gamma\cdot 
e^{-t\tilde{\Omega} \Sigma}\big) &= 16 (g^{-1})^{\mu\nu} 
+ \mathcal{O}(t)\;.
\end{align*}
In terms of the functionals 
$\mathcal{T}_{\mu_1\dots\mu_k} (f) := \mathrm{Tr}_{L^2(\mathbb{R}^4)} 
\big( L_\star(f) \nabla_{\mu_1}\dots\nabla_{\mu_k}
e^{-tH}\big)$ on $\A_\star$ introduced and computed in
Proposition~\ref{Proposition-trace} and the similar functional 
\begin{align*}
\tilde{\mathcal{T}}_{\mu\nu}(f):=\mathrm{Tr}_{L^2(\mathbb{R}^4)}
\big( L_\star(f)\nabla_\mu \tilde{\nabla}_\nu e^{-tH}\big)
\end{align*}
we obtain from Proposition \ref{germ} the trace
$\mathrm{Tr}(e^{-t\mathcal{D}_A^2})$ as follows:
\begin{align}
\mathrm{Tr}\big(e^{-t\mathcal{D}_A^2}\big)
&= e^{-t M^2} \bigg\{16 \cosh^4(\tilde{\Omega} t) \mathrm{Tr}(e^{-tH})
-2t\, \mathcal{T}\big( 16 V_{A,\phi} \big) 
-2t \,  \mathcal{T}_\mu \big(32 i A^\mu\big) 
\nonumber
\\
&+t^2 \,\mathcal{T}\Big(16 V_{A,\phi}\star V_{A,\phi}
+ 16 (g^{-1})^{\mu\nu} D_\mu \phi \star \overline{D_\nu\phi}
+ 8 (g^{-1})^{\mu\rho} (g^{-1})^{\nu\sigma} F^A_{\mu\nu}F^A_{\rho\sigma}
\nonumber
\\*
& \qquad\qquad + 
32i (g^{-1})^{\mu\nu} A_\mu \star \partial_\nu V_{A,\phi}
+16 \tilde{\Omega} \Theta^{\mu\nu}F^A_{\mu\nu}
\Big)
\nonumber
\\
& +t^2\,
\mathcal{T}_\mu\Big(32i A^\mu \star V_{A,\phi}+
32i V_{A,\phi} \star A^\mu
- 64 (g^{-1})^{\nu\rho} A^\nu \star \partial_\rho A^\mu
\Big)
\nonumber
\\
& + 
t^2 \,\mathcal{T}_{\mu\nu}\big(- 64 A^\mu \star  A^\nu \big) 
\nonumber
\\
& -\frac{t^3}{3} \mathcal{T}_{\mu\nu}\Big( 
-4 i ((V_{A,\phi} \star \partial^\mu  A^\nu-
\partial^\mu  A^\nu \star V_{A,\phi} )
\nonumber
\\
& \qquad\qquad
+64 A^\mu (g^{-1})^{\rho\sigma}\partial_\rho\partial_\sigma A^\nu  
+128 A_\rho (g^{-1})^{\rho\sigma} \partial_\sigma \partial^\mu A^\nu
\nonumber
\\
& \qquad\qquad
-64 \big( V_{A,\phi}\star A^\mu\star A^\nu
+ A^\mu\star V_{A,\phi}\star A^\nu
+ A^\mu\star A^\nu\star V_{A,\phi}\big)
\nonumber
\\
& \qquad\qquad
-128i (g^{-1})^{\rho\sigma}
\big( A_\rho{\star} (\partial_\sigma A^\mu) {\star} A^\nu
+ A_\rho\star A^\mu {\star}
(\partial_\sigma A^\nu) + A^\mu{\star} A_\rho {\star}
(\partial_\sigma A^\nu)\big)
\Big)
\nonumber
\\
& -\frac{t^3}{3} \mathcal{T}_{\mu\nu\rho}\Big(
128 A^\mu \star  \partial^\nu A^\rho
-128i A^\mu\star A^\nu \star A^\rho 
\Big)
\nonumber
\\
& -\frac{t^3}{3}
\tilde{\mathcal T}_{\mu\nu} \big( 128 i 
\tilde{\Omega}^2 \Theta^{\rho\nu} 
A^\mu \star  A_\rho\big)
\nonumber
\\
&+ \frac{t^4}{12}
\mathcal{T}_{\mu\nu\rho\sigma}\Big(
-256 A^\mu \star \partial^\nu\partial^\rho A^\sigma
+ 512i A^\mu \star A^\nu \star \partial^\rho A^\sigma
\nonumber
\\
&\qquad\qquad + 256i A^\mu \star 
(\partial^\nu A^\rho) \star A^\sigma 
+ 256 A^\mu \star 
A^\nu \star A^\rho \star A^\sigma 
\Big)\bigg\}
\nonumber
\\
&+\bigg\{ A_\mu \mapsto B_\mu\;,~F^A_{\mu\nu}\mapsto F^B_{\mu\nu}\;,~
V_{A,\phi} \mapsto V_{B,\phi}\;,~D_\mu\phi \leftrightarrow 
\overline{D_\mu\phi} \bigg\}
+ \mathcal{O}(\sqrt{t})\;.
\label{trH1}
\end{align}
The relevant traces have been computed in Proposition \ref{Proposition-trace}.
A similar procedure gives
\[
\tilde{\mathcal{T}}_{\mu\nu}(f)=
\Big(\frac{\tilde{\Omega}}{2\pi\sinh(2\tilde{\Omega}t)}\Big)^2 
\int_{\mathbb{R}^4} dz \; \sqrt{\det g} \,f(z) \Big(
\tilde{\mathcal{N}}_{\mu\nu}
+ Z_\mu \tilde{Z}_\nu\Big) e^{-\tilde{\Omega}\tanh(\tilde{\Omega}t)
 \langle z,gz\rangle}\;,
\]
with $\tilde{Z}_\nu:= -2i (\Theta^{-1} z)_\nu 
- 2\tilde{\Omega}\tanh(\tilde{\Omega}t) (g z)_\nu$
and $\tilde{\mathcal{N}}_{\mu\nu}:=2i (\Theta^{-1}g^{-1})_{\mu\nu}
+i\tilde{\Omega}^2 (g \Theta )_{\mu\nu}
- \tilde{\Omega}\tanh(\tilde{\Omega}t) \delta_{\mu\nu}
$. This shows that the contribution of
$\tilde{\mathcal{T}}_{\mu\nu}(f)$ is suppressed with
$\mathcal{O}(t)$. 
Inserting these traces into (\ref{trH1}) we arrive at
\begin{align}
\label{trH-final}
&\mathrm{Tr}(e^{-t\mathcal{D}_A^2})
\\
&=\frac{2}{\tilde{\Omega}^{4}} t^{-4} 
-\frac{2 M^2}{\tilde{\Omega}^{4}} t^{-3} 
+ \Big(\frac{M^4}{\tilde{\Omega}^4}+\frac{8}{3\tilde{\Omega}^2}\Big) 
t^{-2}
- \Big(\frac{M^6}{3\tilde{\Omega}^4}+\frac{8 M^2}{3\tilde{\Omega}^2} \Big) 
t^{-1}
+ \Big(\frac{52}{45} +\frac{M^8}{12\tilde{\Omega}^4}
+\frac{4 M^4}{3\tilde{\Omega}^2} \Big) 
\nonumber
\\
&-t^{-1}\frac{2-2 M^2 t}{\pi^2} \int d^4 x \sqrt{\det g}
\bigg\{ \phi\star \bar{\phi} +M(\phi + \bar{\phi})
+ \tilde{\Omega}^2 \big( \langle X_A, g X_A\rangle_\star 
-\langle x, g x\rangle_\star \big)
\nonumber\\
&\qquad\qquad\qquad\qquad\qquad +\bar{\phi}
\star \phi +M(\phi + \bar{\phi})
+ \tilde{\Omega}^2 \big( \langle X_B, g X_B\rangle_\star 
-\langle x, g x\rangle_\star \big)
\bigg\}
\nonumber
\\
&+\frac{1}{\pi^2} \int d^4x \sqrt{\det g}  
\bigg\{
2 (g^{-1})^{\mu\nu} D_\mu \phi \star \overline{D_\nu\phi}
\nonumber
\\
&+ \big(\phi\star \bar{\phi} +M(\phi + \bar{\phi}) 
+ \tilde{\Omega}^2 \langle X_A, g X_A\rangle_\star \big)^2
- \big(\tilde{\Omega}^2 \langle x, g x\rangle_\star \big)^2
\nonumber
\\
&+\big(\bar{\phi} \star \phi +M(\phi + \bar{\phi})
+ \tilde{\Omega}^2 \langle X_B, g X_B\rangle_\star \big)^2
- \big(\tilde{\Omega}^2 \langle x, g x\rangle_\star \big)^2
\nonumber
\\
&
+\Big( \frac{1}{2} (g^{-1})^{\mu\rho}(g^{-1})^{\nu\sigma}
- \frac{1}{6} (g^{-1}+\tilde{\Omega}^2 \Theta g\Theta)^{\mu\rho}
(g^{-1}+\tilde{\Omega}^2 \Theta g\Theta)^{\nu\sigma}
\Big) 
\big(F^A_{\mu\nu} \star F^A_{\rho\sigma}
+F^B_{\mu\nu} \star F^B_{\rho\sigma}\big)\bigg\}
\nonumber
\\
&+ \mathcal{O}(\sqrt{t})\;,\nonumber
\end{align}
where
\begin{align*}
X^\mu_A (x)  := x^\mu +\Theta^{\mu\nu} A_\nu\;,\qquad  
\langle X,gY\rangle_\star := g_{\mu\nu} X^\mu \star Y^\nu\;.
\end{align*}
To reduce (\ref{trH1}) to (\ref{trH-final}) we have used
\begin{itemize}\itemsep 0pt
\item[--]  the traciality (\ref{tracial}) of the Moyal product 
and the resulting cyclicity under the integral,

\item[--] integration by parts where appropriate, 

\item[--]  $x^\mu\star f =
 \frac{1}{2}\{x^\mu,f\}_\star 
+  \frac{1}{2} [x^\mu,f]_\star
= \frac{1}{2}\{x^\mu,f\}_\star 
 + i \Theta^{\mu\nu} \partial_\nu f$ where appropriate,

\item[--]  symmetries and antisymmetries in the indices.
\end{itemize}
The matrix $g^{-1}+\tilde{\Omega}^2 \Theta g\Theta$ appearing in front
of the curvature term in (\ref{trH-final}) can be equivalently written
as 
\[
g^{-1}+\tilde{\Omega}^2 \Theta g\Theta
= g^{-1}+ 4g (1-g^{-1})= g^{-1}(1-2g)^2\;.
\]
Since $g^{-1}\geq 1$ according to (\ref{metric}), hence $0\leq g\leq
1$, we have $0\leq g^{-1}+\tilde{\Omega}^2 \Theta g\Theta \leq
g^{-1}$, showing that the matrix in front of the curvature term in
(\ref{trH-final}) is strictly positive. It is minimal for
$g=\frac{1}{2}$, i.e.\ $\tilde{\Omega}^2 (\Theta^t
\Theta)_{\mu\nu}=4\delta_{\mu\nu}$.  The curvature $F_{\mu\nu}$ can also be
expressed in terms of the covariant coordinates,
$[X_A^\mu,X_A^\nu]_\star = i\Theta^{\mu\nu}+ i \Theta^{\mu\rho}
\Theta^{\nu\sigma} F^A_{\rho\sigma}$.

\smallskip

With  the moments $ \chi_{-n}
=\int_0^\infty ds \;s^{n-1}\chi(s)$ of the ``characteristic function'' and
$\chi_0=\chi(0)$, we identify the spectral action (\ref{specaction-gen})
as 
\begin{align}
&S_\Lambda(\mathcal{D}_A)
=\frac{2\Lambda^8}{\tilde{\Omega}^{4}} \chi_{-4}
-\frac{2 M^2 \Lambda^6}{\tilde{\Omega}^{4}} \chi_{-3}
+ \Big(\frac{M^4\Lambda^4}{\tilde{\Omega}^4}
+\frac{8\Lambda^4}{3\tilde{\Omega}^2}\Big) \chi_{-2}
- \Big(\frac{M^6\Lambda^2}{3\tilde{\Omega}^4}
+\frac{8 M^2\Lambda^2}{3\tilde{\Omega}^2} \Big)  \chi_{-1} 
\nonumber
\\
&\qquad \qquad + \Big(\frac{52}{45} +\frac{M^8}{12\tilde{\Omega}^4}
+\frac{4 M^4}{3\tilde{\Omega}^2} \Big) \chi_0
\nonumber
\\
&+\frac{\chi_0}{\pi^2} \int d^4x \sqrt{\det g}  
\bigg\{
2 (g^{-1})^{\mu\nu} D_\mu \phi \star \overline{D_\nu\phi}
\nonumber
\\
&+ \Big(\phi\star \bar{\phi} +M(\phi + \bar{\phi}) 
+ \tilde{\Omega}^2 \langle X_A, g X_A\rangle_\star 
+M^2 - \frac{\chi_{-1}}{\chi_0}\Lambda^2 \Big)^2
- \Big(\tilde{\Omega}^2 \langle x, g x\rangle_\star 
+M^2 - \frac{\chi_{-1}}{\chi_0}\Lambda^2 \Big)^2
\nonumber
\\
&+\Big(\bar{\phi} \star \phi +M(\phi + \bar{\phi})
+ \tilde{\Omega}^2 \langle X_B, g X_B\rangle_\star 
+M^2 - \frac{\chi_{-1}}{\chi_0}\Lambda^2 \Big)^2
- \Big(\tilde{\Omega}^2 \langle x, g x\rangle_\star 
+M^2 - \frac{\chi_{-1}}{\chi_0}\Lambda^2 \Big)^2
\nonumber
\\
&
+\Big( \frac{1}{2} (g^{-1})^{\mu\rho}(g^{-1})^{\nu\sigma}
- \frac{1}{6} (g^{-1}+\tilde{\Omega}^2 \Theta g\Theta)^{\mu\rho}
(g^{-1}+\tilde{\Omega}^2 \Theta g\Theta)^{\nu\sigma}
\Big) 
\big(F^A_{\mu\nu} \star F^A_{\rho\sigma}
+F^B_{\mu\nu} \star F^B_{\rho\sigma}\big)\bigg\}
\nonumber
\\
& + \mathcal{O}(\Lambda^{-1})\;.
\label{SpecAct}
\end{align}
The final result (\ref{SpecAct}) for the spectral action agrees, up to
typos, with the result obtained in \cite{Grosse:2007jy}. We recall
that with the cumbersome computational method of \cite{Grosse:2007jy}
it was only possible to identify the part of the spectral action at
most bilinear in the gauge fields $A,B$. By gauge-invariant completion
it was argued that the total spectral action has to be
(\ref{SpecAct}).  In \cite{Grosse:2007jy} the $\Theta$-matrix was
chosen as $\Theta= \theta \left(\begin{array}{cc} i\sigma_2 & 0 \\ 0 &
   i\sigma_2\end{array}\right)$. In terms of
$\Omega:=\frac{\theta\tilde{\Omega}}{2}$ this choice leads to
$(g^{-1})^{\mu\nu}=(1+\Omega^2)\delta^{\mu\nu}$ and $\sqrt{\det
 g}=\frac{1}{(1+\Omega^2)^2}$. Up to the global factor of 2 due to
the real structure, the few differences in the
prefactors\footnote{These are $
 \Big(\frac{(1+\Omega^2)^2}{2}-\frac{(1-\Omega^2)^4}{6(1+\Omega^2)^2}\Big)$
 versus
 $\Big(\frac{(1-\Omega^2)^2}{2}-\frac{(1-\Omega^2)^4}{3(1+\Omega^2)^2}\Big)
$ in \cite{Grosse:2007jy}
 in front of $F_{\mu\nu}F^{\mu\nu}$ and $\frac{(1+\Omega^2)}{2}$
 versus $\frac{1}{2}$ in \cite{Grosse:2007jy} in front of
 $(D\phi)^2$.}  are easily identified as typos in
\cite{Grosse:2007jy}.  We finish by a brief discussion of the spectral
action:

\begin{itemize}

\item The square of covariant
coordinates $X_A,X_B$ combines with the Higgs field $\phi$ to a 
non-trivial potential.
This was not noticed in \cite{deGoursac:2007gq,Grosse:2007dm}. 
We observe here a much deeper unification of the continuous geometry
described by Yang-Mills fields and discrete geometry described by the
Higgs field than previously in almost-commutative geometry. 

\item The coefficient in front of the
Yang-Mills action is strictly positive for any real-valued
$\tilde{\Omega}$. In the
bosonic model of \cite{deGoursac:2007gq,Grosse:2007dm} there was only
the analogue of the negative part, which leads to problems with the
field equations.

Unlike the scalar model renormalized in \cite{Grosse:2004yu} where
$\Omega=\frac{\theta \tilde{\Omega}}{2}$ can by Langmann-Szabo duality
be restricted to $\Omega\in[0,1]$, the full spectral action
(\ref{SpecAct}) does not have a distinguished frequency parameter
$\tilde{\Omega}>0$.

\item  The action (\ref{SpecAct}) is invariant under
gauge transformations 
\begin{align*}
\phi+M &\mapsto u_A\star(\phi+M)\star \overline{u_B}\;, &
X_{A\mu} &\mapsto u_A\star X_{A\mu}
\star \overline{u_A}\;, &
X_{B\mu} &\mapsto u_B\star X_{B\mu}\star \overline{u_B}\;, 
\end{align*}
where $u_A,u_B\in \mathcal{U}(\mathcal{A}_\star \oplus \mathbb{C})$
are unital elements of the minimal unitization. 

\item For any value of the free parameter $\frac{M^2\chi_0}{\Lambda^2
   \chi_{-1}}$, the action contains $(A,B,\phi)$-linear terms which
 lead to a complicated vacuum which is \emph{not} attained at
 vanishing $A,B,\phi$. Since $A,B,\phi$ are Schwartz functions, the
 \emph{formal} vacuum solution $X_A=0=X_B$ and
 $\phi+M=\sqrt{\frac{\chi_{-1}}{\chi_0}} \Lambda$ is excluded. An
 enlargement of Schwartz class function to e.g.\ polynomially bounded
 functions does not help either, because then we are not allowed to
 expand the Gau\ss{}ian
 $e^{-\tilde{\Omega}\tanh(\tilde{\Omega}t)\langle x,gx\rangle_\star}$
 in $t$, making the spectral action different from 
 (\ref{SpecAct}). 

\item 
If we \emph{formally} regard $\phi+M,X_A,X_B$ as dynamical variables
of the model, then (\ref{SpecAct}) can be viewed as
translation-invariant with respect to 
\begin{align*}
\phi(x)+M &\mapsto \phi(x{+}a)+M\;, &
X_{A}(x) &\mapsto  X_{A}(x{+}a) \;, &
X_{B}(x) &\mapsto  X_{B}(x{+}a) \;.
\end{align*}
This would clear away a frequent objection against the
renormalizable $\phi^4_4$-models, breaking of translation
invariance. However, this transformation leaves the space of 
Schwartz class functions for $A,B,\phi$, so that translation
invariance remains broken in the consistent
spectral action. 

\item The vacuum part of the spectral action is finite. In
 general, the heat kernel expansion for non-compact spectral triples
 is ill-defined, so that a spatial regularization of the operator
 trace is unavoidable. See e.g.\ \cite{Gayral:2004ww}. The oscillator
 potential is one of many possibilities. We want to advertise the
 point of view that if one takes the spectral action principle
 serious, \emph{the spatial regularization is part of the geometry}.
 The removal of the spatial regularization must be carefully studied.
 In general, we should expect that other limiting procedures such as
 those of quantum field theory make it impossible to remove the
 regularization (UV/IR).

\end{itemize}

\begin{appendix}
\section{Locally compact noncommutative spin manifolds} 
\label{LCNSM}
\begin{Definition}
\label{Def:ST}
A {\tt\itshape non-compact spectral triple} is given by the data
$(\mathcal{A},\mathcal{B},
\mathcal{H},\mathcal{D},J,\Gamma,\boldsymbol{c})$ satisfying
conditions 0-6 given below. The data consist of a 
non-unital algebra $\mathcal{A}$ acting faithfully (via a representation
denoted by $\pi$) by bounded operators on the Hilbert space $\mathcal{H}$;
 a preferred unitization $\mathcal{B}$ of $\mathcal{A}$
acting by bounded operators on the same Hilbert space; and an essentially 
self-adjoint unbounded operator $\mathcal{D}$ on $\mathcal{H}$ such that
$[\mathcal{D},\pi(a)]$ extends to a bounded operator for any $a\in
\mathcal{B}$. The spectral triple is said to be
{\tt\itshape even} if there exists a 
$\mathbb{Z}_2$-grading operator $\Gamma$ on $\mathcal{H}$ satisfying 
$\Gamma^2=1$, for which 
$\mathcal{B}$ is even and   $\mathcal{D}$ is odd. 
The spectral triple is said to be {\tt\itshape real} if there exists an 
antiunitary operator $J$ on $\mathcal{H}$ which satisfies conditions 4
and 5 below.
\begin{enumerate}
\item[0.] {\tt\itshape Compactness.}
 \\[\smallskipamount]
 The operator $\pi(a)(\mathcal{D}-\lambda)^{-1}$ is compact for all
 $a\in \mathcal{A}$ and $\lambda$ in the resolvent set of
 $\mathcal{D}$.
 \\[\smallskipamount]
 For any $a \in \mathcal{B}$, both $\pi(a)$ and
 $[\mathcal{D},\pi(a)]$ belong to $\bigcap_{n =1}^\infty \mathrm{dom}
 (\delta^n)$, with $\delta (T):=[\langle\mathcal{D}\rangle,T]$ and
 $\langle \mathcal{D}\rangle:=(\mathcal{D}^2+1)^{\frac{1}{2}}$.

For any element $b$ of the algebra $\Psi_0(\mathcal{A})$ generated by
$\delta^n(\pi(\mathcal A))$ and $\delta^n([\mathcal{D},\pi(\mathcal
A)])$, the function $\zeta_b(z):=\mathrm{Tr}(b
\langle\mathcal{D}\rangle^{-z})$ is well defined and holomorphic for
$\Re(z)$ large and analytically continues to $\mathbb{C}\setminus
\mathrm{Sd}$ for some discrete set $\mathrm{Sd} \subset \mathbb{C}$
(the dimension spectrum). Moreover the dimension spectrum is said to
be simple if all the poles are simples, finite if there is
$k\in\mathbb N$ such that all the poles are order at most $k$ and if
not, infinite.

\item {\tt\itshape Metric dimension.}
 \\*[\smallskipamount]
 For the metric dimension $d:=\sup \{ \Re(z),\, z \in \mathrm{Sd}\}$,
 the operator $ \pi(a) \langle \mathcal{D}\rangle^{-d}$ belongs to
 the Dixmier ideal $\mathcal{L}^{1,\infty}(\mathcal H)$ for any $a
 \in \mathcal{A}$. Moreover, for any Dixmier trace, the map
 $\A_+ \ni a \mapsto \mathrm{Tr}_\omega( \pi(a) \langle
 \mathcal{D}\rangle^{-d})$ is non-vanishing.

\item {\tt\itshape Finiteness.}
\\*[\smallskipamount]
The algebra $\mathcal{A}$ and its preferred unitization
$\mathcal{B}$ are pre-$C^\star$-algebras, i.e.\
each one is a $\star$-subalgebra of some $C^*$-algebra and stable
under holomorphic functional calculus. 

The space of smooth spinors $\displaystyle
\mathcal{H}^\infty:=\bigcap_{k=0}^\infty \mathcal{H}^k$, with 
$\mathcal{H}^k:=
\mathrm{dom}(\mathcal{D}^k) $ completed with norm 
$\|\xi\|_k^2:=\|\xi\|^2 + \|\mathcal{D}^k \xi\|^2$, 
is a finitely generated projective $\mathcal{A}$-module
$p \mathcal{A}^m$, for some $m \in \mathbb{N}$ and some projector
$p=p^2=p^* \in M_m(\mathcal{B})$. The composition of the
Dixmier trace with the induced
hermitian structure $\langle ~,~\rangle_{\mathcal{A}} :
\mathcal{H}^\infty \times \mathcal{H}^\infty \to \mathcal{A}$
coincides with the scalar product $(~,~)$ on $\mathcal{H}^\infty$,
\[
(\xi,\eta) = \mathrm{Tr}_\omega\Big( \langle \xi,\eta
\rangle_{\mathcal{A}} \,\langle\mathcal{D}\rangle^{-d}\Big)\;,\qquad 
\xi,\eta \in \mathcal{H}^\infty\;.
\]

\item {\tt\itshape Reality.}
\\*[\smallskipamount]
The operator $J$ defines a real structure of KO-dimension 
$k\in \mathbb{Z}_8$. This means 
\begin{align}
J^2 &= \varepsilon\;, & J\mathcal{D} &= \varepsilon' \mathcal{D}J\;, &
J\Gamma &= \varepsilon'' \Gamma J \quad \text{(even case)}
\end{align}                        
with signs $\varepsilon ,\varepsilon' ,
\varepsilon'' \in \{ -1,1\}$ given as a function of $k \mod 8$ by
\[
\begin{array}{|c| r r r r r r r r|} \hline 
k  &0 &1 &2 &3 &4 &5 &6 &7 \\ \hline \hline
\varepsilon  &1 & 1&-1&-1&-1&-1& 1&1 \\
\varepsilon' &1 &-1&1 &1 &1 &-1& 1&1 \\
\varepsilon''&1 &{}&-1&{}&1 &{}&-1&{} \\  \hline
\end{array}
\]
Additionally, the action $\pi$ of $\mathcal{B}$ on $\mathcal{H}$ 
satisfies the commutation rule $[\pi(f),\pi^o(g)] = 0$ for all $f,g
\in\mathcal{B}$, where $\pi^o(g)=J\pi(g^*)J^{-1}$ is the 
action of the opposite algebra $\mathcal{B}^o$.

\item {\tt\itshape First order.}
\\*[\smallskipamount]
$[[\mathcal{D},\pi(f)],\pi^o(g)] = 0$ for all $f,g \in \mathcal{B}$.

\item  {\tt\itshape Orientability.}
\\*[\smallskipamount]
Whenever the metric dimension $d$ is an integer, there is a 
Hochschild $d$-cycle $\boldsymbol{c}$ on
$\mathcal{B}$ with values in 
$\mathcal{B}\otimes \mathcal{B}^o$, i.e.\ a finite sum
of terms $(a_0\otimes b_0) \otimes a_1 \otimes \dots \otimes a_d$. Its
representation $\pi_\D( \boldsymbol{c})$ with 
$\pi_\D((a_0\otimes b_0) \otimes a_1 \otimes \dots \otimes a_d)
:= \pi(a_0)J\pi(b_0^*)J^{-1} [\mathcal{D},\pi(a_1)]\cdots 
[\mathcal{D},\pi(a_d)]$ satisfies $\pi_\D(\boldsymbol{c})^2=1$ and 
defines the volume form on $\mathcal{A}$, i.e.\ 
\[
\phi_{\boldsymbol{c}}(f_0,\dots,f_{d}) = 
\mathrm{Tr}_\omega\big(\pi_\D(\boldsymbol{c}) \pi(f_0) 
[\mathcal{D},\pi(f_1)]\cdots  
[\mathcal{D},\pi(f_d)]
\langle\mathcal{D}\rangle^{-d}\big)
\]
provides a non-vanishing Hochschild $d$-cocycle $\phi_d$ on
$\mathcal{A}$. 
\end{enumerate}
\end{Definition}

\end{appendix}

\section*{Acknowledgments}

R.W. would like to thank Harald Grosse for the long-term collaboration
which initiated this paper through the preprint \cite{Grosse:2007jy}. 
Both authors would like to thank Alan Carey for stimulating
discussions.

\end{document}